\documentclass{article}

\PassOptionsToPackage{numbers}{natbib}

\usepackage[final]{neurips_2022}

\usepackage[utf8]{inputenc} 
\usepackage[T1]{fontenc}    
\usepackage{hyperref}       
\usepackage{url}            
\usepackage{booktabs}       
\usepackage{amsfonts}       
\usepackage{nicefrac}       
\usepackage{microtype}      
\usepackage{xcolor}         

\usepackage{amsmath}
\usepackage{amsthm}
\usepackage{amssymb}
\usepackage{bbm}
\usepackage{stmaryrd}
\usepackage{subfig}
\usepackage{mathrsfs}
\usepackage[all]{xy}
\usepackage{enumitem}
\usepackage{tikz}
\usepackage{yhmath}
\usetikzlibrary{arrows,calc,shapes,decorations.pathreplacing}
\usepackage{algorithmic}
\usepackage{comment}

\newtheorem{thm}{Theorem}
\newtheorem{prop}[thm]{Proposition}
\theoremstyle{definition}
\newtheorem{df}[thm]{Definition}

\newtheorem{coro}[thm]{Corollary}
\newtheorem{lem}[thm]{Lemma}

\theoremstyle{remark}

\newtheoremstyle{condition}{}{}{}{}{\bfseries}{.}{.5em}{#1 \thmnote{#3}}
\theoremstyle{condition}
\newtheorem*{cond}{Condition}

\makeatletter
\newtheorem*{rep@theorem}{\rep@title}
\newcommand{\newreptheorem}[2]{%
\newenvironment{rep#1}[1]{%
 \def\rep@title{#2 \ref{##1}}%
 \begin{rep@theorem}}%
 {\end{rep@theorem}}}
\makeatother

\newreptheorem{lem}{Lemma}

\newcommand{\fonction}[5]{
 \begin{array}{rrcl}
	#1: & #2 & \longrightarrow & #3 \\
    & #4 & \longmapsto & #5 \end{array}
    }
    
\newcommand{\varfonction}[4]{
 \begin{array}{rcl}
	 #1 & \longrightarrow & #2 \\
     #3 & \longmapsto & #4 \end{array}
    }
    
\newcommand{\RR}{\mathbb{R}}
\newcommand{\NN}{\mathbb{N}}
\newcommand{\ZZ}{\mathbb{Z}}
\newcommand{\One}{\boldsymbol{1}}
\newcommand{\Par}{\RR^E \times \RR^B}
\newcommand{\ParS}{\left(\RR^E \times \RR^B\right) \backslash S}
\newcommand{\ParF}{\RR^{F_{\theta}} \times \RR^B}

\newcommand{\ParSX}{\left(\RR^E \times \RR^B\right) \backslash (S \cup \Delta_X)}
\newcommand{\D}{\diag}
\newcommand{\p}{p}
\newcommand{\thF}{\tau_{\theta}}
\newcommand{\lb}{\llbracket}
\newcommand{\rb}{\rrbracket}
\newcommand{\rsim}{\overset{R}{\sim}}
\renewcommand{\Im}{\operatorname{Im}}

\DeclareMathOperator{\Ker}{Ker}

\DeclareMathOperator{\diag}{Diag}
\DeclareMathOperator{\Id}{Id}

\DeclareMathOperator{\sgn}{sign}
\DeclareMathOperator{\rk}{rank}

\title{Local Identifiability of Deep ReLU Neural Networks: the Theory}

\author{Joachim Bona-Pellissier$^{ab}$$^{*}$, Fran\c{c}ois Malgouyres$^{b}$,  François Bachoc$^{b}$ \\
        \small Institut de Math\'ematiques de Toulouse ; UMR 5219\\
         Universit\'e de Toulouse ; CNRS \\
        \small $^{a}$  UT1, F-31042 Toulouse, France\\
       \small $^{b}$ UPS, F-31062 Toulouse Cedex 9, France\\\\\\
       \small $^{*}$Corresponding author: Joachim Bona-Pellissier; \tt{joachim.bona-pellissier@univ-toulouse.fr}
}

\begin{document}

\maketitle

\begin{abstract}Is a sample rich enough to determine, at least locally, the parameters of a neural network? To answer this question, we introduce a new local parameterization of a given deep ReLU neural network by fixing the values of some of its weights. This allows us to define local lifting operators whose inverses are charts of a smooth manifold of a high dimensional space. The function implemented by the deep ReLU neural network composes the local lifting with a linear operator which depends on the sample. We derive from this convenient representation a geometric necessary and sufficient condition of local identifiability. Looking at tangent spaces, the geometric condition provides: 1/ a sharp and testable necessary condition of identifiability and 2/ a sharp and testable sufficient condition of local identifiability. The validity of the conditions can be tested numerically using backpropagation and matrix rank computations. 
\end{abstract}

\section{Introduction}

\subsection{Context and motivations}\label{Context and motivations-sec}

Neural networks are famous for their capacity to perform complex tasks in a wide variety of domains such as image classification \cite{krizhevsky2012imagenet}, object recognition \cite{redmon2016you,ren2015faster}, speech recognition \cite{6296526,sak2014long,hannun2014deep}, natural language processing \cite{DBLP:journals/corr/abs-1301-3781,mikolov2010recurrent,kalchbrenner2013recurrent}, anomaly detection \cite{pinto2021deep} or climate sciences \cite{adewoyin2021tru}. 

The following properties of the parameters of neural networks have recently drawn attention: identifiability, inverse stability and stable recovery; from weaker to stronger. Let $f_{\theta}(X)$ be the outputs of a network parameterized by the parameters $\theta$, for given inputs $X$. Global identifiability means that if $f_{\theta}(X) = f_{\tilde{\theta}}(X)$ then $\theta = \tilde{\theta}$, up to identified invariances, for instance neuron permutation and rescaling for ReLU networks. Local identifiability restricts this analysis for $\theta$ and $\tilde{\theta}$ sufficiently close. Then, inverse stability means that the distance between $\theta$ and $\tilde{\theta}$ (up to invariances) is bounded by a function of the distance between $f_{\theta}(X)$ and $f_{\tilde{\theta}}(X)$. Finally, stable recovery consists in obtaining an algorithm to approximately recover $\theta$ from a noisy version of $f_{\theta}(X)$, with quantitative guarantees. In all cases, we must distinguish between statements for $X$ being a finite list of inputs, in which case we would like $X$ to be small, and for infinite $X$ (for instance determining $\theta$ from the entire function $f_{\theta}$).

Identifiability from finite $X$, which is the focus of this paper, is important for different reasons. In the first place, model extraction attacks for neural networks have been a growing topic over the last years. Indeed, some algorithms are able to recover in practice the parameters of a neural network from queries \cite{carlini2020cryptanalytic, pmlr-v119-rolnick20a}. This can be a concern since neural network providers may wish to keep these parameters secret, for security \cite{kurakin2016adversarial}, for privacy \cite{fredrikson2015model, carlini2019secret}, or for intellectual property \cite{zhang2018protecting}.

A way of preventing such a recovery can be by guaranteeing that identifiability does not hold, that is to check that a necessary condition of identifiability is not met. On the opposite side, guaranteeing that identifiability holds is interesting in the position of an attacker. If the attacker has access to $X$, to $f_\theta(X)$, and is able to compute a $\tilde{\theta}$ such that $f_{\tilde{\theta}}(X) = f_\theta(X)$, the question then becomes: does this guarantee that $\tilde{\theta} = \theta$ or shall the attacker expand $X$ with new queries? The attacker needs a sufficient condition of identifiability.

Another important motivation for identifiability is having a better understanding and control of neural networks. Indeed, if the learning sample has the form $(X, f_\theta(X))$, with $\theta$ the parameters of a teaching network, global identifiability from $X$ means that the global minimizer of the empirical risk is unique. In this case, if the global minimizer is reached, there will typically be no variability due to the optimization parameters (choice of the algorithm, number of epochs,...) and to stochasticity (for stochastic optimizers). Even if very recent works on double descent phenomena, e.g. \cite{belkin2020two}, highlight a benefit of overparameterization (thus absence of identifiability) for increasing prediction performances, a user may be interested in a small enough number of parameters to retain identifiability, if the loss of performance is mild compared to overparameterization.

Note that, of course, global identifiability is more relevant than local identifiability to the above motivations. This work nevertheless focuses on local identifiability, which is a necessary condition for global identifiability, and which analysis can be a first step to analyzing global identifiability. Local identifiability is also arguably insightful on the geometry of the relationship between the parameter space of $\theta$ and its image $\{f_{\theta}(X), \ \theta \text{ varies}\}$.
Note that most existing identifiability, inverse stability and stable recovery results (see the next section) are also local.

\subsection{Existing work on identifiability, inverse stability, stable recovery and attacks}
\paragraph{Identifiability:}

Even though it has regained interest recently, the question of identifiability for neural networks is not new. Indeed, in the 1990s, some positive results of identifiability for networks with smooth activation functions ($\mathrm{tanh}$, logistic sigmoid or Gaussian for instance) have been established \cite{sussmann1992uniqueness, albertini1993uniqueness, kuurkova1994functionally, kainen1994uniqueness, fefferman1994reconstructing}. These results are mainly theoretical, they concern activation functions which are not the most used nowadays (in particular, they do not apply to ReLU networks), and assume full knowledge of the function $f_\theta$ implemented by the network, which is impossible in practice.

When it comes to ReLU, for shallow \cite{petzka2020symmetries, stock:tel-03208517} as well as deep \cite{phuong2020functional, bona2021parameter} neural networks, some positive results of identifiability have been recently established. They show that under some conditions on the architecture and parameters of the network, the function implemented by the network uniquely characterizes its parameters, up to neuron permutation and rescaling operations. Although they apply to ReLU networks, these results share a limitation with those of previous paragraph: they assume the function implemented by the network to be known on the whole input space, or at least on an open subset of it.

As far as we know, there exists only one identifiability result for deep ReLU networks assuming the knowledge of $f_\theta$ on a \emph{finite} sample only. \citet{stock:hal-03292203} give a theoretical condition for the existence of a finite set which locally identifies the parameters of a deep neural network. It is an \emph{existence} result: it does not concretely provide such a finite set, nor does it allow to test local identifiability for any finite sample, as we propose in this work. The construction in \cite{stock:hal-03292203} shares similarities with previous works on deep structured matrix factorization \cite{MalgouyresLandsbergITW, MalgouyresLandsberg_long, malgouyres2020stable}. The present article also lies in this line of research. 

\paragraph{Inverse stability and stable recovery:}

Closely related to identifiability are the topics of inverse stability and stable recovery of the parameters of a network. Some negative \cite{petersen2020topological} as well as positive \cite{berner2019degenerate, MalgouyresLandsbergITW, MalgouyresLandsberg_long, malgouyres2020stable} results of inverse stability exist. The articles \cite{MalgouyresLandsbergITW, MalgouyresLandsberg_long, malgouyres2020stable} examine the case of structured networks with the identity as activation function. Only \cite{malgouyres2020stable} considers a finite $X$.
The authors of \cite{berner2019degenerate} consider a general class of networks amongst which ReLU networks, but the result only holds for one-hidden-layer neural networks. Furthermore this result also requires the knowledge of $f_\theta$ on a whole domain.

Several stable recovery algorithms have also been proposed, for one-hidden-layer neural networks in a first place, for smooth activation function \cite{fu2020guaranteed}, as well as ReLU in the fully-connected case \cite{ge2018learning, zhang2019learning, zhong2017recovery, zhou2021local} or in the convolutional case \cite{brutzkus2017globally, zhang2020improved}. These references consider a finite $X$ but provide a large sample complexity under which a smartly constructed initialization followed by a first order algorithm allows to stably recover the parameters of the network.

For deep networks, some stable recovery algorithms also exist, for instance for Heavyside activation function \cite{arora2014provable}, or for only recovering the first layer with sparsity assumptions \cite{sedghi2014provable} in the ReLU case, but to the best of our knowledge there does not exist any algorithm recovering fully a deep ReLU network from a finite sample.

\paragraph{Model inversion attacks:} For deep ReLU networks, when one has full access to the function implemented by the network, a practical algorithm \cite{pmlr-v119-rolnick20a} sequentially constructs a sample $X$ and approximately recovers the architecture and the parameters modulo permutation and rescaling. Similarly, formulating the problem as a cryptanalytic problem, \cite{carlini2020cryptanalytic} reconstructs a functionally equivalent network with fewer requests. As mentioned in Section \ref{Context and motivations-sec}, these two references are related to identifiability, but consider a different setting. In this article we consider an arbitrary given $X$, while they work mostly on its construction.

\subsection{Contributions}

\paragraph{1/} We establish a necessary and sufficient geometric condition of local identifiability from a finite sample $X$ for deep fully-connected ReLU networks. The condition is that the intersection between an affine space and a smooth manifold is reduced to a single point. See Figure \ref{geometric condition} for an illustration. 

\paragraph{2/} Considering tangent spaces, we then provide a computable necessary condition of local identifiability from a finite sample X. Since global identifiability implies local identifiability, it is also a computable necessary condition of global identifiability. 

\paragraph{3/} We also establish a computable sufficient condition of local identifiability, which is close to the necessary condition. To the best of our knowledge, these are the first testable conditions of local identifiability for any finite input sample. In particular, \cite{stock:hal-03292203} provides a theoretical condition equivalent to the existence of a finite sample for which local identifiability holds, but does not provide the sample explicitly, nor does it characterize local identifiability for any arbitrary sample.

\begin{figure}
    \centering
    \begin{tabular}{lr}
    \includegraphics[scale=0.8]{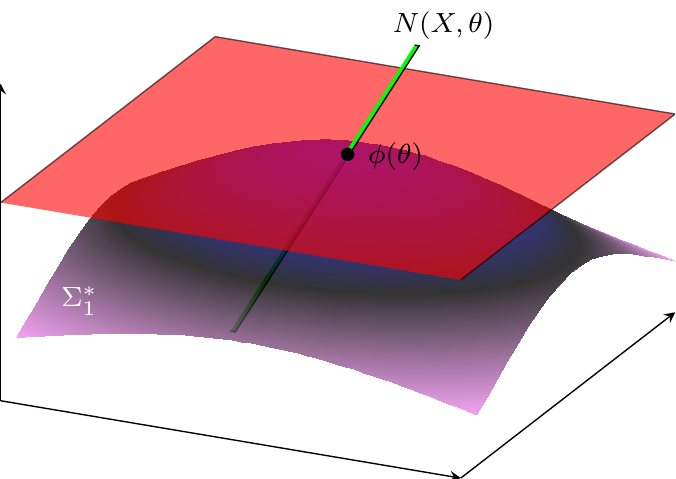} \hspace{20pt}
    &
    \includegraphics[scale=0.8]{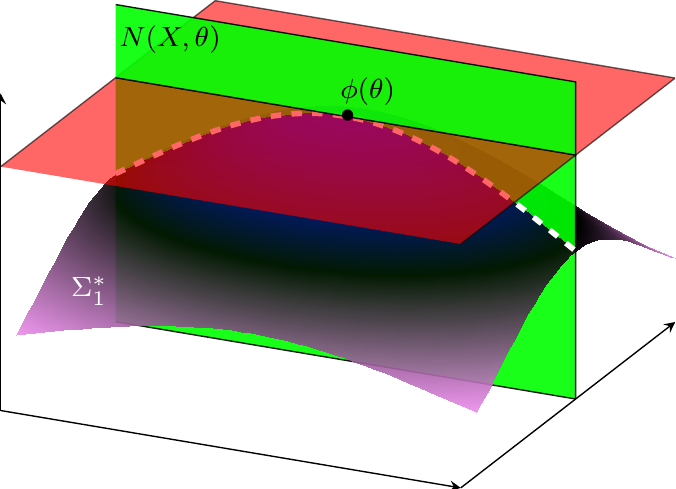}
    \end{tabular}
    \caption{The local intersection between the affine space $N(X,\theta)$ (in green) and the smooth manifold $\Sigma_1^*$ (color gradient). We also represent in red the tangent space to $\Sigma_1^*$ at $\phi(\theta)$. Left: The identifiable case. The intersection is reduced to $\{\phi(\theta)\}$. Right: The non identifiable case. The intersection, represented with a dashed white line, is not reduced to $\{\phi(\theta)\}$.}
    \label{geometric condition}
\end{figure}

\paragraph{4/} To prove these results, we develop geometric tools which can be of independent interest for theoretically understanding deep ReLU networks as well as for possible applications. Namely, we introduce local reparameterizations $\rho_\theta$ of the network by fixing some weight values as constants. Building on these local parameterizations, we introduce local lifting operators $\psi^\theta$ and we decompose the function implemented by the network $f_\theta(x)$ as a composition of $\psi^\theta$, which only depends on the parameters, and a piecewise constant operator $\alpha$ which depends on $\theta$ and the inputs $x^i$. For almost any parameterization $\theta$, the operator $\alpha$ is constant in a neighborhood of $\theta$ and consists in applying a linear function to $\psi^\theta$. We show that in fact, the operators $\psi^\theta$ are the inverses of coordinate charts of a smooth manifold $\Sigma_1^*$, contained in a high dimensional space. We find $\Sigma_1^*$ to be of particular interest in representing geometrically some properties of the network parameters (in particular to establish 1/, 2/ and 3/ above).

\subsection{Overview of the article}

This work is structured as follows. We start by introducing basic tools and already known results, and we state the definition of local identifiability in Section \ref{ReLU networks, lifting operator and rescaling of the parameters-sec-main}. We then introduce the local parameterizations $\rho_\theta$ and the set $\Sigma_1^*$, and we show that the latter is a smooth manifold in Section \ref{The smooth manifold Sigma_1^*-sec-main}. This allows us to state our main results in Section \ref{main results-sec-main}, that is the geometric and the numerically testable conditions of local identifiability. Finally we discuss in Section \ref{Computations-sec-main} the numerical computations needed to test the latter conditions. All the proofs are provided in the appendices.

\section{ReLU networks, lifting operator and rescaling of the parameters}\label{ReLU networks, lifting operator and rescaling of the parameters-sec-main}


\subsection{ReLU networks}\label{ReLU networks-sec-main}

Let us introduce our notations for deep fully-connected ReLU networks. In this paper, a network is a graph $(E,V)$ of the following form. 
\begin{itemize}
\item $V$ is a set of neurons, which is divided in $L+1$ layers, with $L \geq 2$: $V = (V_l)_{l \in \lb 0, L \rb}$.

$V_0$ is the input layer, $V_L$ the output layer and the layers $V_l$ with $1 \leq l \leq L-1$ are the hidden layers. Using the notation $|C|$ for the cardinal of a finite set $C$, we denote, for all $l \in \lb 0, L \rb$, $N_l = |V_l |$ the size of the layer $V_l$.
\item $E$ is the set of all oriented edges $v \rightarrow v'$ between neurons in consecutive layers, that is
\[E = \{ v \rightarrow v' , \ v \in V_l, v' \in V_{l+1}, \text{for } l \in \lb 0 , L-1 \rb \}.\]
\end{itemize}
A network is parameterized by weights and biases, gathered in its parameterization $\theta$, with
\[\theta = \left( (w_{v \rightarrow v'})_{v \rightarrow v' \in E}, (b_v)_{v \in B}\right) \quad \in \Par ,\]
where $B = \bigcup_{l=1}^L V_l$. It is also convenient to consider the weights and biases in matrix/vector form: for a given $\theta$, we denote, for $l \in \lb 1, L \rb$, 
\[W_l = (w_{v \rightarrow v'})_{v' \in V_l, v \in V_{l-1}} \in \RR^{N_l \times N_{l-1}} \qquad \text{and} \qquad b_l = (b_v)_{v \in V_l} \in \RR^{N_l}.\]
When dealing with two parameterizations $\theta$ and $\tilde{\theta} \in \Par$, we take as a convention that $w_{v \rightarrow v'}$ and $b_v$ as well as $W_l$ and $b_l$ denote the weights and biases associated to $\theta$, and $\tilde{w}_{v \rightarrow v'}$ and $\tilde{b}_v$ as well as $\widetilde{W}_l$ and $\tilde{b}_l$ denote those associated to $\tilde{\theta}$.

The activation function, denoted $\sigma$, is always ReLU: for any $p \in \NN^*$ and any vector \mbox{$x = (x_1, \dots , x_p)^T \in \mathbb{R}^p$}, it is defined as $\sigma(x) = (\max(x_1,0), \dots , \max(x_p,0))^T$.

For a given $\theta$, we define recursively $f_l : \RR^{V_0} \rightarrow \RR^{V_l}$ (we omit the dependency in $\theta$ in the notation for simplicity), for $l \in \lb 0, L \rb$, by
\begin{itemize}
\item $ \forall x \in \RR^{V_0}$, \qquad $f_0(x) = x$ ;
\item $\forall l \in \lb 1, L-1 \rb$, $\forall x \in \RR^{V_{0}}$,
\qquad $f_l(x) = \sigma \left(W_l f_{l-1}(x) + b_l \right)$;
\item $\forall x \in \RR^{V_{0}}$, \qquad $f_L(x) = W_L f_{L-1}(x) + b_L$ .
\end{itemize}

We define the function $f_\theta : \RR^{V_0} \rightarrow \RR^{V_L}$ implemented by the network of parameter $\theta$ as $f_\theta = f_L$.

\subsection{The lifting operator \texorpdfstring{$\phi$}{} and the activation operator \texorpdfstring{$\alpha$}{} }\label{the lifting operator-sec-main}

For a fixed $x \in \mathbb{R}^{V_{0}}$, the value of $f_{\theta}(x)$ is a non-linear function of $\theta$. The goal of this section is to obtain a higher-dimensional representation of $\theta$, that will be written $\phi(\theta)$, and such that $f_{\theta}(x)$ is locally a linear function of $\phi(\theta)$. This will be achieved with Proposition \ref{fundamental prop of alpha-main}. The function $\phi$ is called a lifting operator, a wording borrowed from category theory and commonly used in compressed sensing and dictionary learning, for instance in \cite{candes2015phase}. The components of $\phi(\theta)$ will be associated to paths in the neural network. Linearity in Proposition \ref{fundamental prop of alpha-main} will correspond to summing over these paths.

We now introduce the paths notations.
 For all $l \in \lb 0, L-1 \rb$, we define
\[\mathcal{P}_l = V_l \times \dots \times V_{L-1},\]
which is the set of all paths in the network starting from layer $l$ and ending in layer $L-1$. We consider an additional element $\beta$ which can be interpreted as an empty path and whose role will be clear once $\phi$ has been defined and Proposition \ref{fundamental prop of alpha-main} stated. We define
\[\mathcal{P} = \left( \bigcup_{l=0}^{L-1} \mathcal{P}_l \right) \cup \{ \beta \}.\]

In a similar way to \cite{stock:hal-03292203}, we can now define the above-mentioned `lifting operator'
\begin{equation}
    \fonction{\phi}{\RR^E \times \RR^B}{\RR^{\mathcal{P}\times V_L}}{\theta}{(\phi_{p,v}(\theta))_{p \in \mathcal{P}, v \in V_L}}
\end{equation}
by:
\begin{itemize}
\item for all $l \in \llbracket 0, L-1 \rrbracket$ and all $p =(v_l, \dots, v_{L-1}) \in \mathcal{P}_l$, and for all $v_L \in V_L$,
\[\phi_{p,v_L}(\theta) = \begin{cases}\prod_{l'=0}^{L-1} w_{v_{l'} \rightarrow v_{l'+1}} & \text{if } l=0 \\ 
b_{v_l} \prod_{l'=l}^{L-1} w_{v_{l'}  \rightarrow v_{l'+1}}& \text{if } l \geq 1;\end{cases}\]
\item for $p = \beta$ and $v_L \in V_L$, \quad $\phi_{\beta,v_L}(\theta) = b_{v_L}$.
\end{itemize}

To define the activation operator, we first define, for all $l \in \lb 1 , L-1 \rb$, all $v \in V_l$, all $\theta \in \Par$ and $x \in \RR^{V_0}$,
\[a_v(x,\theta) = \begin{cases} 1 & \text{if } (W_{l} f_{l-1}(x)+b_l)_v \geq 0 \\ 0 & \text{otherwise,}\end{cases}\]
which is the activation indicator of neuron $v$. We then define the `activation operator'
\begin{equation}\label{definition of alpha}
    \fonction{\alpha}{\RR^{V_0} \times \left( \RR^E \times \RR^B \right)}{\RR^{1 \times \mathcal{P}}}{(x, \theta)}{(\alpha_p(x, \theta))_{p \in \mathcal{P}}}
\end{equation}
by:
\begin{itemize}
\item for all $l \in \llbracket 0, L-1 \rrbracket$ and all $p =(v_l, \dots, v_{L-1}) \in \mathcal{P}_l$:
\[\alpha_{p}(x, \theta) = \begin{cases} x_{v_0} \prod_{l'=1}^{L-1} a_{v_{l'}}(x, \theta) & \text{if } l=0 \\
\prod_{l'=l}^{L-1} a_{v_{l'}}(x, \theta) & \text{if } l \geq 1; \end{cases}\]
\item for $p = \beta$, \quad $\alpha_{\beta}(x, \theta) = 1$.
\end{itemize}

We then have the announced linear representation of the function $f_\theta$ implemented by the network.

\begin{prop}\label{fundamental prop of alpha-main}
For all $\theta \in \RR^E \times \RR^B$ and all $x \in \RR^{V_0}$, \quad $f_{\theta}(x)^T = \alpha(x, \theta) \phi(\theta)$.
\end{prop}

This result, which is proven in Appendix \ref{the lifting operator-sec}, is for instance also stated in \citep[Sec.~4]{stock:hal-03292203} with slightly different notations. Note that each component of the vector $f_{\theta}(x)$ above is written as a sum over a (very large) number of paths.

Let us reformulate Proposition \ref{fundamental prop of alpha-main} with several inputs. We consider, for some $ n \in \NN^*$, some given inputs $x^i \in \RR^{V_0}$, with $i \in \lb 1 , n \rb$. We denote by $X \in \RR^{n \times V_0}$ the matrix whose lines are the transpose $(x^i)^T$ of the inputs. For all $ \theta \in \Par$, we denote by $f_\theta(X) \in \RR^{n \times V_L}$ the matrix whose lines are the transpose $f_\theta(x^i)^T$ of the corresponding outputs. We also denote by $\alpha(X, \theta) \in \RR^{n \times \mathcal{P}}$ the matrix whose lines are the line vectors $\alpha(x^i, \theta)$. Using Proposition \ref{fundamental prop of alpha-main} for all the $x^i$, we have the relation
\begin{equation}\label{fundamental prop of alpha-matrix version-main}
f_\theta(X) = \alpha(X, \theta) \phi(\theta).
\end{equation} 

We prove in Appendix \ref{the lifting operator-sec} the next proposition, which states that $\theta \mapsto \alpha(X, \theta)$ is piecewise constant.
\begin{prop}\label{alpha_X is piecewise-constant-main}
For all $n \in \NN^*$, for all $X \in \RR^{n \times V_0}$, the mapping
\[\fonction{\alpha_X}{\Par}{\RR^{n \times \mathcal{P}}}{\theta}{\alpha(X, \theta)}\]
is piecewise-constant, with a finite number of pieces.
Furthermore, the boundary of each piece has Lebesgue measure zero. We call $\Delta_X$ the union of all these boundaries. The set $\Delta_X \subset \Par$ is closed and has Lebesgue measure zero.
\end{prop}

As discussed before, for a given $X \in \RR^{n \times V_0}$, when studying the function $\theta \mapsto f_\theta(X)$, Proposition \ref{alpha_X is piecewise-constant-main} alongside \eqref{fundamental prop of alpha-matrix version-main} shows that on a piece over which $\alpha_X$ is constant, $f_\theta(X)$ depends linearly on $\phi(\theta)$. Since $\Delta_X$ is closed with measure zero, for almost all $\tilde{\theta} \in \Par$, there exists a neighborhood of $\tilde{\theta}$ over which $\alpha_X$ is constant. As noted for instance by \citet[Sec.~2]{stock:hal-03292203}, for any $\theta$ in such a neighborhood, we thus have
\begin{equation}\label{f_theta(X) - f_(tilde(theta))(X) = alpha(X,theta)(phi(theta) - phi(tilde(theta))}
    f_{\theta}(X) - f_{\tilde{\theta}}(X) = \alpha(X, \tilde{\theta}) \left( \phi(\theta) - \phi(\tilde{\theta})\right). 
\end{equation}
Hence, studying $\phi$ will allow us to understand better how $f_\theta(X)$ locally depends on $\theta$.

\subsection{Invariant rescaling operations on \texorpdfstring{$\theta$}{}}\label{Rescaling-sec-main}

Some well-known rescaling operations on the parameters $\theta$ do not affect the value of $\phi(\theta)$. Before detailing them, let us define, for all $t \in \RR$, the sign indicator $\sgn(t)$ as $1, 0$ or $-1$ depending on whether $t>0$, $t=0$ or $t< 0$ respectively. For any $\theta \in \Par$, we then define
\[\sgn(\theta) = \Big( (\sgn(w_{v \rightarrow v'})_{v \rightarrow v' \in E}, (\sgn(b_v))_{v \in B}\Big) \in \{-1, 0, 1 \}^E \times \{-1, 0, 1 \}^B.\]
We can now describe the rescaling operations.
\begin{df} \label{def:rescaling}
Let $\theta \in \Par$ and $\tilde{\theta} \in \Par$.
\begin{itemize}
\item We say that $\theta$ is equivalent to $\tilde{\theta}$ modulo rescaling, and we write $\theta \rsim \tilde{\theta}$ iff there exists a family of vectors \mbox{$(\lambda^0, \dots, \lambda^L) \in (\RR^*)^{V_0} \times \dots \times (\RR^*)^{V_L}$}, with $\lambda^0 = \One_{V_0}$ and $\lambda^L = \One_{V_L} $, such that, for all $l \in \lb 1, L \rb$,
\begin{equation}\label{equivalence def-1-main}
\begin{cases}W_l = \D(\lambda^l) \widetilde{W}_l \D(\lambda^{l-1})^{-1} \\
b_l = \D(\lambda^l) \tilde{b}_l.\end{cases}
\end{equation}

\item We say that $\theta$ is equivalent to $\tilde{\theta}$ modulo positive rescaling, and we write $\theta \sim \tilde{\theta}$ iff
\[\theta \overset{R}{\sim} \tilde{\theta} \quad \text{and} \quad \sgn(\theta) = \sgn(\tilde{\theta}).\]
\end{itemize}
\end{df}

For all $l \in \lb 1, L \rb$, to satisfy \eqref{equivalence def-1-main} is equivalent to satisfy, for all $(v_{l-1},v_l) \in V_{l-1} \times V_l$, 
\begin{equation}\label{equivalence def-2-main}
\begin{cases}
w_{v_{l-1} \rightarrow v_{l}} = \frac{\lambda^l_{v_l}}{\lambda^{l-1}_{v_{l-1}}} \tilde{w}_{v_{l-1} \rightarrow v_l} \\
b_{v_l} = \lambda^l_{v_l} \tilde{b}_{v_l}.
\end{cases}
\end{equation}

The relations $\rsim$ and $\sim$ are equivalence relations on the set of parameters $\Par$. The equivalence modulo positive rescaling $\sim$ is a well-known invariant for ReLU networks \cite{stock:tel-03208517, stock:hal-03292203, bona2021parameter, neyshabur2015path, yi2019positively}. We have indeed the following property: if $\theta \sim \tilde{\theta}$, for all $x \in \RR^{V_0}$, 
\begin{equation} \label{eq:rescaling:f:invariant}
    f_\theta(x) = f_{\tilde{\theta}}(x).
\end{equation}

One of the interests of the operator $\phi$ is that it captures this invariant, as described by \citet[Sec. 2.4]{stock:hal-03292203}. Propositions \ref{equivalence implies equal phi-main} and \ref{equal phi implies resc equivalence-main} are similar to their results and are restated here and proven in Appendix \ref{the lifting operator-sec} for completeness. Indeed, combining the definition of $\phi$ with \eqref{equivalence def-2-main}, we have the following property.

\begin{prop}\label{equivalence implies equal phi-main}
For all $\theta, \tilde{\theta} \in \Par$, we have
\[\theta \overset{R}{\sim} \tilde{\theta} \quad \Longrightarrow \quad \phi(\theta) = \phi( \tilde{\theta}),\]
and thus in particular
\[\theta \sim \tilde{\theta} \quad \Longrightarrow \quad \phi(\theta) = \phi(\tilde{\theta}).\]
\end{prop}

The reciprocal of Proposition \ref{equivalence implies equal phi-main} holds provided we exclude some degenerate cases. Let us denote, for any $l \in \lb 1, L-1 \rb$ and any $v \in V_l$, by $w_{\bullet \rightarrow v}$ the vector $(w_{v' \rightarrow v})_{v' \in V_{l-1}} \in \RR^{V_{l-1}}$ and by $w_{v \rightarrow \bullet}$ the vector $(w_{v \rightarrow v'})_{v' \in V_{l+1}} \in \RR^{V_{l+1}}$.  We define the following set, which is close to the notion of `non admissible parameter' in \cite{stock:hal-03292203}:
\[S = \{ \theta \in \Par, \exists v \in V_1 \cup \dots \cup V_{L-1}, \ w_{v \rightarrow \bullet} = 0 ~~ \text{ or } ~~ (w_{\bullet \rightarrow v}, b_v) = (0,0) \}.\]
When $w_{v \rightarrow \bullet} = 0$, all the outward weights of $v$ are zero. When $(w_{\bullet \rightarrow v}, b_v) = (0,0)$, all the inward weights as well as the bias of $v$ are zero, so for any input the information flowing through neuron $v$ is always zero. In both cases, the neuron $v$ does not contribute to the output and could be removed from the network without changing the function $f_\theta$.
Since the set $S$ is a finite union of linear subspaces of codimension larger than 1, it is closed and has Lebesgue measure zero. We can thus exclude the degenerate cases in $S$ without loss of generality. Proposition \ref{equal phi implies resc equivalence-main} states that the reciprocal of Proposition \ref{equivalence implies equal phi-main} holds over $\ParS$.
\begin{prop}\label{equal phi implies resc equivalence-main}
For all $\theta \in \ParS $, for all $\tilde{\theta} \in \Par$, 
\[\phi(\theta) = \phi( \tilde{\theta}) \quad \Longrightarrow \quad \theta \overset{R}{\sim} \tilde{\theta}.\]
\end{prop}

\subsection{Local identifiability}

We have now introduced all the concepts used in the formal definition of `local identifiability'.

\begin{df} \label{def:local:identif}
Let $X \in \RR^{n \times V_0}$ and $\theta \in \RR^E \times \RR^B$. We say that $\theta$ is \emph{locally identifiable from $X$} if there exists $\epsilon > 0$ such that for all $\tilde{\theta} \in \Par $,  if $\|\theta - \tilde{\theta}\|_{\infty} < \epsilon$,
\[
f_{\theta}(X) = f_{\tilde{\theta}}(X) \quad \Longrightarrow \quad \theta \sim \tilde{\theta}.
\]
\end{df}

\section{The smooth manifold \texorpdfstring{$\Sigma_1^*$}{}}\label{The smooth manifold Sigma_1^*-sec-main}

We explained in the previous section that studying $\phi$ allows to better understand how the output $f_\theta(X)$ locally depends on $\theta$. The image of $\phi$ is of particular interest in this study and is the subject of this section. We define
\[\Sigma_1^* = \{ \phi(\theta), \ \theta \in \ParS \}.\]
The main result of this section, Theorem \ref{Sigma_1^* is a smooth manifold-main}, states that $\Sigma_1^*$ is a smooth manifold. This result is a key element of the article. Indeed, it allows to consider tangent spaces to $\Sigma_1^*$, and by doing so, to linearize the geometric characterization of Theorem \ref{local identifiability using Sigma1*-main} illustrated in Figure \ref{geometric condition}. Instead of considering the intersection between a smooth manifold and an affine space as in Theorem \ref{local identifiability using Sigma1*-main}, this indeed allows to consider the intersection between two affine spaces, which can be characterized with rank computations as in Theorems \ref{Necessary condition-main} and \ref{Sufficient condition-main}.

To show this result, we need local injectivity. In this aim, let us consider a fixed $\theta$ and analyze the functions $u \mapsto f_{\theta+u}(X)$ and $u \mapsto \phi(\theta+u)$ for $u$ around $0$. We can select $N_1+\dots+N_{L-1}$ scalar scaling parameters (each in a neighborhood of $1$), and use them to ``rescale'' $\theta+u$ as in Definition \ref{def:rescaling}, leaving $f_{\theta+u}(X)$ and $\phi(\theta+u)$ unchanged (\eqref{eq:rescaling:f:invariant} and Proposition \ref{equivalence implies equal phi-main}). Locally, at first order, this means that there are 
$N_1+\dots+N_{L-1}$ linear combinations of $u$ which leave $f_{\theta+u}(X)$ and $\phi(\theta + u)$ invariant. In order to obtain injectivity with respect to $u$, locally around $0$, we will fix $N_1+\dots+N_{L-1}$ components of $u$ as follows.

\begin{figure}
\centering
\begin{tabular}{lr}
\includegraphics[scale=1]{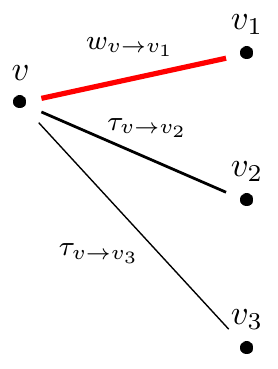}
&
\hspace{35pt}
\includegraphics[scale=1]{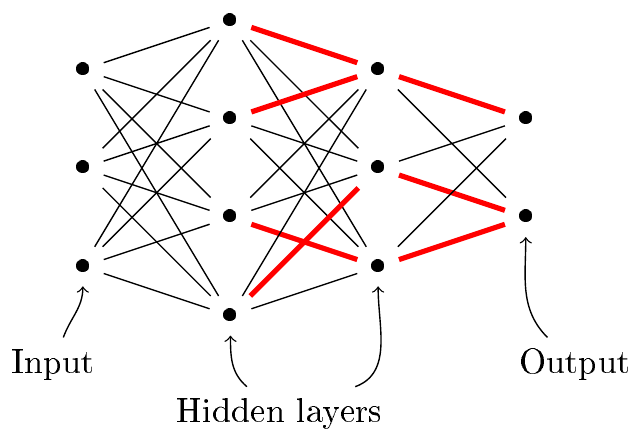}
\end{tabular}
\caption{Left: The outward edges of a hidden neuron $v$ and their weights. In this example, $v_1 = s^\theta_{\max}(v)$, so the weight of the edge in red, $v \rightarrow v_1$, has its value fixed as $w_{v \rightarrow v_1}$. The weights of the remaining edges, $\tau_{v \rightarrow v_2}$ and $\tau_{v \rightarrow v_3}$, are free to vary. Right: In red, all the edges whose weights are fixed. The remaining edges, in black, constitute the set $F_\theta$.}\label{test}
\end{figure}

For each neuron $v$ in a hidden layer, we choose the outward edge $v \rightarrow v'$ whose weight $w_{v \rightarrow v'}$ has largest (absolute) value (if there are several such edges, we choose one arbitrarily). We denote by $s^\theta_{\max}(v)$ such a neuron $v'$. For each neuron $v$ in a hidden layer $V_l$, there is exactly one neuron $s^\theta_{\max}(v)$ in the layer $V_{l+1}$, and one corresponding edge $v \rightarrow s^\theta_{\max}(v)$. See Figure \ref{test} for an illustration. We will set to $0$ the components of $u$ corresponding to all the edges of the form $v \rightarrow s^\theta_{\max}(v)$. Intuitively, it will not limit the set of functions $f_{\tilde\theta}$, in the vicinity of $f_{\theta}$; but will permit to obtain a one-to-one correspondence between $u$ and $f_{\theta+u}$.

More precisely, let us denote by $F_\theta \subset E$ the set of remaining edges, which is formally defined as\footnote{Note, in the definition of $F_\theta$, the index $l$ starting at $l=1$ and not $l=0$.}
\begin{equation}\label{definition of F_theta-main}
F_\theta = E \ \backslash \ \left( \bigcup_{l=1}^{L-1} \Big\{ (v, s^\theta_{\max}(v)), v \in V_l \Big\} \right).
\end{equation}

The mapping from the space of restricted parameters $\ParF$ to the parameter space $\Par$ locally around $\theta$ is simply given by
 the following application
\begin{equation}\label{definition of rho_theta-main}
\fonction{\rho_{\theta}}{\ParF}{\Par}{\tau}{\tilde{\theta} \quad \text{such that }\begin{cases}\forall (v, v') \in F_{\theta}, \quad \tilde{w}_{v \rightarrow v'} = \tau_{v \rightarrow v'} \\ \forall (v,v') \in E \backslash F_{\theta}, \quad \tilde{w}_{v \rightarrow v'} = w_{v \rightarrow v'} \\ \forall v \in B, \quad \tilde{b}_{v} = \tau_v. \end{cases}}
\end{equation}
In particular, if we define $\thF \in \ParF$ by $(\thF)_{v \rightarrow v'} = w_{v \rightarrow v'}$ and $(\thF)_v = b_v$, we have $\rho_\theta(\thF) = \theta$.
The function $\rho_{\theta}$ is affine and injective. We define 
\begin{equation}\label{definition of U_theta-main}
U_{\theta} = \rho_{\theta}^{-1}\left(\ParS\right),
\end{equation}
which is an open set of $\ParF$. We define, for all $\theta \in \ParS$, the local lifting operator
\begin{equation}\label{definition psi^theta-main}
\fonction{\psi^{\theta}}{U_{\theta}}{\RR^{\mathcal{P}\times V_L}}{\tau}{\phi \circ \rho_{\theta}(\tau).}
\end{equation}

One can show that $\psi^\theta$ is $C^\infty$ and that it is a homeomorphism from $U_\theta$ onto its image (see the proofs in Appendix \ref{The smooth manifold structure of Sigma_1^*-sec}), which we denote $V_\theta$ and is thus an open subset of $\Sigma_1^*$ (with the topology induced on $\Sigma_1^*$ by the standard topology on $\RR^{\mathcal{P}\times V_L}$). In particular, since $\rho_\theta (\thF) = \theta$, we have $\phi(\theta) = \psi^\theta(\thF) \in V_\theta$. We have the following fundamental result that will allow us to consider and make use the tangent spaces of $\Sigma_1^*$.

\begin{thm}\label{Sigma_1^* is a smooth manifold-main}
$\Sigma_1^*$ is a smooth manifold of $\RR^{\mathcal{P} \times V_L}$ of dimension 
\[|F_{\theta}| + |B| = N_0 N_1 + N_1 N_2 + \dots + N_{L-1}N_L + N_L,\]
and the family $(V_\theta, (\psi^\theta)^{-1})_{\theta \in \ParS}$ is an atlas.
\end{thm}
Theorem \ref{Sigma_1^* is a smooth manifold-main} is proven in Appendix \ref{The smooth manifold structure of Sigma_1^*-sec}. Besides being key in Section \ref{main results-sec-main}, Theorem \ref{Sigma_1^* is a smooth manifold-main} (both the smooth manifold nature of $\Sigma_1^*$ and the explicit atlas $(V_\theta,(\psi^\theta)^{-1})_{\theta \in \ParS}$) may also be considered of more general independent interest. To our knowledge, such a result has not been established elsewhere in the literature. Notice that, as announced, despite the use of restricted parameters in $\ParF$, we can represent the {\em whole} tangent space at any point of $\Sigma_1^*$. The only consequence of the restriction is the uniqueness of the representation of the elements of tangent spaces.


\section{Main results: necessary and sufficient conditions for local identifiability}\label{main results-sec-main}

The main results of this paper rely on the decomposition \eqref{f_theta(X) - f_(tilde(theta))(X) = alpha(X,theta)(phi(theta) - phi(tilde(theta))} introduced in Section \ref{ReLU networks, lifting operator and rescaling of the parameters-sec-main}. To reformulate \eqref{f_theta(X) - f_(tilde(theta))(X) = alpha(X,theta)(phi(theta) - phi(tilde(theta))}, let us introduce the linear operator $A(X,\theta)$, which simply corresponds to the matrix product with $\alpha(X,\theta)$:
\[\fonction{A(X,\theta)}{\RR^{\mathcal{P} \times V_L}}{\RR^{n \times V_L}}{\eta}{\alpha(X,\theta) \eta,}\]
where $\alpha(X, \theta) \eta$ is the matrix product between $\alpha(X, \theta) \in \RR^{n \times \mathcal{P}}$ and $\eta \in \RR^{\mathcal{P} \times V_L}$. The operator $A(X,\theta)$ inherits the properties of $\alpha(X,\theta)$, in particular those stated in Proposition \ref{alpha_X is piecewise-constant-main}. Using $A(X,\theta)$, the relation \eqref{f_theta(X) - f_(tilde(theta))(X) = alpha(X,theta)(phi(theta) - phi(tilde(theta))} satisfied by $\tilde{\theta}$ in the neighborhood of $\theta$ becomes
\begin{equation}\label{A(X,theta) phi(theta) = f_theta (X)-main}
f_{\theta}(X) - f_{\tilde{\theta}}(X) = A(X, \theta) \cdot \left( \phi(\theta) - \phi(\tilde{\theta})\right).
\end{equation}

Let us also define the affine space (set-sum of a fixed point and a vector space)
\[N(X, \theta) = \phi(\theta) + \Ker A(X, \theta).\]

If a parameterization $\tilde{\theta} \in \Par$ is such that $f_{\tilde{\theta}}(X) = f_\theta(X)$ and \eqref{A(X,theta) phi(theta) = f_theta (X)-main} holds, then $ \phi(\theta) - \phi(\tilde{\theta}) \in \Ker A(X, \theta)$, so by definition $\phi(\tilde{\theta}) \in N(X, \theta)$. Since for $\tilde{\theta}$ in the neighborhood of $\theta$, we also have $\phi(\tilde{\theta}) \in \Sigma_1^*$, we see that local identifiability is closely related to the nature of the intersection between the smooth manifold $\Sigma_1^*$ and the affine subspace $N(X, \theta)$. 

Indeed, let us denote by $B_\infty(\phi(\theta), \epsilon) = \{ \eta \in \RR^{\mathcal{P} \times V_L}, \| \phi(\theta) - \eta \|_{\infty} < \epsilon \}$ the ball of center $\phi(\theta)$ and of radius $\epsilon > 0$. We have the following geometric necessary and sufficient condition of local identifiability, which states that local identifiability of $\theta$ holds if and only if the intersection between $\Sigma_1^*$ and $N(X, \theta)$ is locally reduced to the single point $\{\phi(\theta)\}$.

\begin{thm}\label{local identifiability using Sigma1*-main}
For any $X \in \RR^{n \times V_0}$ and $\theta \in \ParSX$, the two following statements are equivalent.
\begin{itemize}
\item[i)]$\theta$ is locally identifiable from $X$.

\item[ii)]There exists $\epsilon > 0$ such that $B_\infty(\phi(\theta), \epsilon) \cap \Sigma_1^* \cap N(X,\theta) = \{\phi(\theta)\}$.
\end{itemize}
\end{thm}

Theorem \ref{local identifiability using Sigma1*-main} is proven in Appendix \ref{Conditions of local identifiability-sec}, and is illustrated in Figure \ref{geometric condition}. This geometric condition is crucial for showing the next two results which give testable conditions of identifiability. Theorems \ref{Necessary condition-main} and \ref{Sufficient condition-main} rely on the rank of $A(X,\theta)$ and of another linear operator $\Gamma(X,\theta)$, which we now define. Since, as we said, the function $\psi^\theta$ is $C^\infty$, let us denote by $D{\psi^{\theta}}(\tau): \ParF \rightarrow \RR^{\mathcal{P}\times V_L}$ its differential at the point $\tau$, for any $\tau \in U_\theta$. We define the linear operator $\Gamma(X, \theta) : \ParF \rightarrow \RR^{n \times V_L}$ by
\begin{equation}\label{definition of Gamma}
   \Gamma (X, \theta) = A(X, \theta) \circ D\psi^\theta(\thF).
\end{equation}

We denote $R_{A} = \rk (A(X, \theta))$ and $R_{\Gamma} = \rk (\Gamma (X, \theta))$. Since $\Gamma (X, \theta)$ is defined on $\ParF$, we have $0 \leq R_{\Gamma} \leq |F_{\theta}| + |B|$, and the expression \eqref{definition of Gamma} shows that we also have $0 \leq R_{\Gamma} \leq R_{A}$. We can now define the two following conditions.
\begin{cond}[$C_N$]\label{Condition C_N}
Condition $C_N$ is satisfied by $(\theta, X)$ iif $R_\Gamma < R_A$ or $R_\Gamma = |F_{\theta}| + |B|$.
\end{cond}

\begin{cond}[$C_S$]\label{Condition C_S}
Condition $C_S$ is satisfied by $(\theta, X)$ iif $R_\Gamma = |F_{\theta}| + |B|$.
\end{cond}

The following result states that $C_N$ is necessary for local and therefore global identifiability.

\begin{thm}[Necessary condition of identifiability]\label{Necessary condition-main}
Let $X \in \RR^{n \times V_0}$ and $\theta \in \ParSX$. If $C_N$ is not satisfied, then $\theta$ is not locally identifiable from $X$ (thus not globally identifiable).
\end{thm}


The following result states that $C_S$ is a sufficient condition of local identifiability.

\begin{thm}[Sufficient condition of local identifiability]\label{Sufficient condition-main}
Let $X \in \RR^{n \times V_0}$ and $\theta \in \ParSX$. If $C_S$ is satisfied, then $\theta$ is locally identifiable from $X$.
\end{thm}


Both theorems are proven in Appendix \ref{Conditions of local identifiability-sec}. To discuss these two results, let us point out that the output spaces of $\Gamma(X, \theta)$ and $A(X, \theta)$ have the same dimension, equal to $nN_L$. Each new input adds $N_L$ to this dimension. One can verify that $R_A - R_\Gamma$ is initially $0$ and cannot decrease when new inputs are added.  If a new input leads to $R_A > R_\Gamma$, it can be discarded to preserve $R_A = R_\Gamma$. Moreover, such an input seems unlikely when $R_A < |F_\theta| + |B|$. If the equality $R_\Gamma = R_A$ is enforced, the condition $R_\Gamma = |F_\theta| + |B|$ is both necessary and sufficient. Finally, to satisfy $R_\Gamma = |F_\theta| + |B|$, the dimensions must satisfy $n N_L \geq |F_\theta| + |B|$. The general belief is that the latter is the condition of identifiability since $n N_L$ is the number of scalar measurements and $|F_\theta| + |B|$ is the number of independent free parameters, see Theorem \ref{Sigma_1^* is a smooth manifold-main}.



\section{Checking the conditions numerically}\label{Computations-sec-main}
 
The key benefit of the conditions $C_N$ and $C_S$, compared to the existing literature, is that they can be numerically tested for any fixed finite sample. They need the computation of the rank of two linear operators, namely $\Gamma(X, \theta)$ and $A(X, \theta)$. The operator $\Gamma(X,\theta)$ satisfies the following:

\begin{prop}\label{Gamma is the differential-main}
Let $X \in \RR^{n \times V_0}$ and $\theta \in \ParSX$. The function $ \tau \mapsto f_{\rho_\theta(\tau)}(X)$, for  $\tau \in U_\theta$ is differentiable in a neighborhood of $\thF$, and we denote by $D_\tau f_{\rho_\theta(\tau_\theta)}(X)$ its differential at $\thF$. We have 
\begin{equation}\label{Gamma = D_tau f}
D_\tau f_{\rho_\theta(\tau_\theta)}(X) = \Gamma(X, \theta).
\end{equation}
\end{prop}

The proof of Proposition \ref{Gamma is the differential-main} is in Appendix \ref{Checking the conditions numerically-sec}.
Since the reparameterization with $\rho_\theta$ simply consists in fixing the weights of the edges $v \rightarrow s^\theta_{\max}(v)$ to the value $w_{v \rightarrow s^\theta_{\max}(v)}$, \eqref{Gamma = D_tau f} shows that the coefficients of $\Gamma(X, \theta)$ can be computed by a classic backpropagation algorithm $N_L$ times for each input $x^i$, simply omitting the derivatives with respect to the edges of the form $v \rightarrow s^\theta_{\max}(v)$. An explicit expression of the coefficients of $\Gamma(X, \theta)$ is given in the Appendix \ref{Checking the conditions numerically-sec}.

To be satisfied, $C_S$ needs the dimensions of $\Gamma(X, \theta)$ to satisfy $nN_L \geq |F_\theta| + |B|$. One then needs to compute the rank $R_\Gamma$ of $\Gamma(X, \theta)$, which means computing the rank of a $nN_L \times (|F_\theta| + |B|)$ matrix. Existing algorithms allow to do this with a complexity $O(nN_L (|F_\theta| + |B|)^{\omega-1})$ (up to polylog terms), where $\omega$ is the matrix multiplication exponent and satisfies $\omega < 2.38$ \cite{cheung2013fast}.

When it comes to $C_N$, one needs in addition to know the rank $R_A$ of $A(X, \theta)$, which, as Proposition \ref{rank of A-main} states, requires to compute the rank of $\alpha(X,\theta)$.
\begin{prop}\label{rank of A-main}
Let $X \in \RR^{n \times V_0}$ and $\theta \in \Par$. We have $R_A = N_L \rk \left(\alpha(X,\theta)\right)$.
\end{prop}
The dimensions of $\alpha(X,\theta)$ are sensibly larger, with $|\mathcal{P}|$ columns and $n$ lines, and typically $|\mathcal{P}| >> n$. However it may have some sparsity properties, as its entries consist in products of activation indicators (with possibly one input $x^i_{v_0}$), any one of them being zero causing many entries to vanish. The question of the efficient computation of $R_A$ still needs to be explored and is left as open for future work.

\section{Conclusion} \label{section:CCL}

This paper is the first to characterize local identifiability for deep ReLU networks for any given finite sample, with testable conditions. The practical use of these conditions deserves follow-up research, and so does an extension of our approach to inverse stability. The role of ReLU is crucial in our approach, especially for the necessary condition of local identifiability and with the linear representation (Proposition \ref{fundamental prop of alpha-main}). In the end, from Theorem \ref{Sufficient condition-main} and Proposition \ref{Gamma is the differential-main},  the sufficient condition for local indentifiability is expressed from the Jacobian matrix of the neural network function with respect to its parameters. Extending this to other activation functions than ReLU is an interesting perspective.

\begin{ack}
The authors would like to thank Pierre Stock and Rémi Gribonval for the fruitful discussions around this work, notably regarding the construction of $\phi$ and its link to the question of local identifiability. 

This work has benefited from the AI Interdisciplinary Institute ANITI. ANITI is funded by the French ``Investing for the Future – PIA3” program under the Grant agreement n°ANR-19-PI3A-0004.

The authors gratefully acknowledge the support of the DEEL project.\footnote{\url{https://www.deel.ai/}}

\end{ack}

\bibliographystyle{plainnat}

\bibliography{biblio}

\section*{Checklist}


\begin{enumerate}

\item For all authors...
\begin{enumerate}
  \item Do the main claims made in the abstract and introduction accurately reflect the paper's contributions and scope?
    \answerYes{ The outline in Section 1.4 (“Overview of the article”) provides pointers to where the claimed contributions of the paper are provided.}
  \item Did you describe the limitations of your work?
    \answerYes{Section \ref{Computations-sec-main} acknowledges the open problem of an efficient computation of the rank of $\alpha(X,\theta)$ and Section \ref{section:CCL} describes other remaining open questions.}
  \item Did you discuss any potential negative societal impacts of your work?
    \answerNA{This is a
theoretical/foundation work that adds to the theory and methodology of deep learning.
As for any such contributions, the positive or negative societal impact will depend on
the application case. We do not promote any harmful use of this theory, but we expand
on the existing knowledge.}
  \item Have you read the ethics review guidelines and ensured that your paper conforms to them?
    \answerYes{See previous question.}
\end{enumerate}

\item If you are including theoretical results...
\begin{enumerate}
  \item Did you state the full set of assumptions of all theoretical results?
   \answerYes{All our results explicitly refer to their required assumptions. Some general assumptions that hold throughout the paper are also stated at the beginning.} 
        \item Did you include complete proofs of all theoretical results?
    \answerYes{All the proofs are provided in the supplement.}
\end{enumerate}

\item If you ran experiments...
\begin{enumerate}
  \item Did you include the code, data, and instructions needed to reproduce the main experimental results (either in the supplemental material or as a URL)?
    \answerNA{We did not run experiments.}
  \item Did you specify all the training details (e.g., data splits, hyperparameters, how they were chosen)?
    \answerNA{}
        \item Did you report error bars (e.g., with respect to the random seed after running experiments multiple times)?
    \answerNA{}
        \item Did you include the total amount of compute and the type of resources used (e.g., type of GPUs, internal cluster, or cloud provider)?
    \answerNA{}
\end{enumerate}

\item If you are using existing assets (e.g., code, data, models) or curating/releasing new assets...
\begin{enumerate}
  \item If your work uses existing assets, did you cite the creators?
     \answerNA{We did not use existing assets (code, data or models) nor cure/release new assets (code, data or models).}
  \item Did you mention the license of the assets?
    \answerNA{}
  \item Did you include any new assets either in the supplemental material or as a URL?
   \answerNA{}
  \item Did you discuss whether and how consent was obtained from people whose data you're using/curating?
   \answerNA{}
  \item Did you discuss whether the data you are using/curating contains personally identifiable information or offensive content?
   \answerNA{}
\end{enumerate}

\item If you used crowdsourcing or conducted research with human subjects...
\begin{enumerate}
  \item Did you include the full text of instructions given to participants and screenshots, if applicable?
   \answerNA{We did not use crowdsourcing nor conducted research with human subjects.}
  \item Did you describe any potential participant risks, with links to Institutional Review Board (IRB) approvals, if applicable?
  \answerNA{}
  \item Did you include the estimated hourly wage paid to participants and the total amount spent on participant compensation?
   \answerNA{}
\end{enumerate}

\end{enumerate}


\newpage

\appendix

\section{Notations}

In this section, we define notations, many of which are standard, that are useful in the proofs.

We denote by $\NN$ the set of all natural numbers, including $0$, and by $\NN^*$ the set $\NN$ without $0$. We denote by $\ZZ$ the set of all integers. For any $a,b \in \ZZ$, we denote by $\lb a, b \rb$ the set of all integers $k \in \ZZ$ satisfying $a \leq k \leq b$. For any finite set $A$, we denote by $|A|$ the cardinal of $A$.

For $n, N \in \mathbb{N}^*$, we denote by $\mathbb{R}^N$ the $N$-dimensional real vector space and by $\mathbb{R}^{n \times N}$ the vector space of real matrices with $n$ lines and $N$ columns. For a vector $x = ( x_1, \dots, x_N)^T \in \RR^N$, we use the norm $\| x \|_{\infty} = \max_{i \in \lb 1, N \rb}|x_i|$. For $x \in \mathbb{R}^N$ and $r>0$, we denote $B_\infty(x,r) = \{ y \in \mathbb{R}^N , \| y - x \|_\infty < r \}$. 

For any vector $x = (x_1, \dots, x_N)^T \in \RR^N$, we define $\sgn(x) =(\sgn(x_1), \dots, \sgn(x_N))^T \in \{-1, 0, 1 \}^N$ as the vector whose $i^{th}$ component is equal to
\[\sgn(x_i) =  \begin{cases}1 & \text{if } x_i > 0 \\ 0 & \text{if } x_i = 0 \\ -1 & \text{if } x_i < 0. \end{cases}\]

For any matrix $M \in \mathbb{R}^{n\times N}$, for all $i \in \llbracket 1 , n \rrbracket$, we denote by $M_{i,:}$ the $i^{\text{th}}$ line of $M$. The vector $M_{i,:}$ is a line vector whose $j^{\text{th}}$ component is $M_{i,j}$. Similarly, for $j \in \llbracket 1 , N \rrbracket$, we denote by $M_{:,j}$ the $j^{\text{th}}$ column of $M$, which is the column vector whose $i^{\text{th}}$ component is $M_{i,j}$. For any matrix $M \in \mathbb{R}^{n \times N}$, we denote by $M^T \in \mathbb{R}^{N\times n}$ the transpose matrix of $M$. 

We denote by $\Id_N$ the $N \times N$ identity matrix and by $\One_N$ the vector \mbox{$(1,1, \dots, 1)^T \in \mathbb{R}^N$}. If $\lambda \in \mathbb{R}^N$ is a vector of size $N$, for some $N \in \mathbb{N}^*$, we denote by $\diag(\lambda)$ the $N \times N$ matrix defined by:
\[\diag(\lambda)_{i,j} = \begin{cases} \lambda_i & \text{if } i=j \\ 0 & \text{otherwise.}\end{cases}\] 

If $X$ and $Y$ are two sets and $h : X \rightarrow Y$ is a function, for a subset $A \subset Y$, we denote by $h^{-1}(A)$ the preimage of $A$ under $f$, that is
\[h^{-1}(A)=\{x \in X, h(x) \in A \}.\]
Note that this does not require the function $h$ to be injective.

For any $n,N \in \NN^*$ and any differentiable function $f:\RR^n \rightarrow \RR^N$, for all $x \in \RR^n$, we denote by $Df(x)$ its differential at the point $x$, i.e. the linear application $Df(x):\RR^n \rightarrow \RR^N$ satisfying, for all $h \in \RR^n$,
\[f(x+h) = f(x) + Df(x) \cdot h + o(h).\]
If we denote by $x_j$ and $h_j$ the components of $x$ and $h$, for $j \in \llbracket 1, n \rrbracket$, we have
\[Df(x)\cdot h = \sum_{j=1}^n \frac{\partial f}{\partial x_j}(x) h_j,\]
where for all $j$, $\frac{\partial f}{\partial x_j}(x) \in \RR^N$.
If $f: \RR^n \rightarrow \RR^N$ is a linear application, we denote by $\Ker f$ the set $\{ x \in \RR^n, f(x) = 0\}$, which is a linear subset of $\RR^n$.

\section{The lifting operator \texorpdfstring{$\phi$}{}}\label{the lifting operator-sec}

Let us introduce the notion of `path', extending the definition in Section \ref{the lifting operator-sec-main}. A path is a sequence of neurons $(v_k, v_{k+1}, \dots, v_l) \in V_k \times V_{k+1} \times \dots \times V_l$, for integers $k,l$ satisfying $0 \leq k \leq l \leq L$. In particular, for all $l \in \lb 0, L-1 \rb$, the set $\mathcal{P}_l$ defined in Section \ref{the lifting operator-sec-main} contains all the paths starting from layer $l$ and ending in layer $L-1$. We recall
\[\mathcal{P} = \left( \bigcup_{l=0}^{L-1} \mathcal{P}_l \right) \cup \{ \beta \}.\]

If $k,l,m \in \NN$ are three integers satisfying $0 \leq k < l \leq m \leq L$, and $p = (v_k, \dots, v_{l-1}) \in V_k \times \dots \times V_{l-1}$ and $p' = (v_l, \dots, v_m) \in V_l \times \dots \times V_m$ are two paths such that $p$ ends in the layer preceding the starting layer of $p'$, we define the union of the paths by
\[p \cup p' = ( v_k, \dots, v_{l-1}, v_l, \dots v_m) \in V_k \times \dots \times V_m.\]

Before proving Proposition \ref{fundamental prop of alpha-main}, let us compare briefly our construction to \cite{stock:hal-03292203}. The lifting operator $\phi$ introduced in Section \ref{the lifting operator-sec-main} is similar to the operator $\boldsymbol{\Phi}$ in \cite{stock:hal-03292203}, except that $\boldsymbol{\Phi}$ does not take a matrix form. The operator $\alpha(x,\theta)$ introduced in Section \ref{the lifting operator-sec-main} corresponds partly to the object $\overline{\boldsymbol{\alpha}}(\theta,x)$ in \cite{stock:hal-03292203}. One of the differences is that $\overline{\boldsymbol{\alpha}}(\theta,x)$ does not include any product with $x_{v_0}$ in its entries, as does $\alpha(x, \theta)$.
Finally, a similar statement to Proposition \ref{fundamental prop of alpha-main} and a similar proof can be found in
\cite{stock:hal-03292203}. However, one of the present contributions is to simplify the construction.

Let us now prove Proposition \ref{fundamental prop of alpha-main}, which we restate here. 
\begin{prop}\label{fundamental prop of alpha}
For all $\theta \in \RR^E \times \RR^B$ and all $x \in \RR^{V_0}$,
\[f_{\theta}(x)^T = \alpha(x, \theta) \phi(\theta).\]
\end{prop}

\begin{proof} Let us prove first the following expression, for all $v_L \in V_L$:
\begin{multline}\label{expression of f} f_{\theta}(x)_{v_L} = \Bigg(\sum_{\substack{v_{0} \in V_{0} \\ \vdots \\ v_{L-1} \in V_{L-1}}}   x_{v_0} w_{v_0 \rightarrow v_1} \prod_{l = 1}^{L-1} a_{v_l}(x, \theta) w_{v_{l} \rightarrow v_{l+1}} \Bigg) \\ +\Bigg( \sum_{l=1}^{L-1} \sum_{\substack{v_{l} \in V_{l} \\ \vdots \\ v_{L-1} \in V_{L-1}}}   b_{v_l} \prod_{l' = l}^{L-1} a_{v_{l'}}(x, \theta) w_{v_{l'} \rightarrow v_{l'+1}} \Bigg) + b_{v_L}. \quad
\end{multline}

We prove this by induction on the number $L$ of layers of the network.

Initialization ($L=2$). Let $v_2 \in V_2$.

\begin{equation*}
\begin{aligned}
f_{\theta}(x)_{v_2} & = (W_2)_{v_2,:} \  \sigma \left( W_1 x + b_1 \right) + b_{v_2} \\
& = \left( \sum_{v_1 \in V_1} w_{v_1 \rightarrow v_2} \left[ \sigma \left( W_1 x + b_1 \right) \right]_{v_1} \right) + b_{v_2} \\
& = \left( \sum_{v_1 \in V_1} w_{v_1 \rightarrow v_2} \sigma \left( (W_1)_{v_1,:} \ x + b_{v_1} \right) \right) + b_{v_2} \\
& = \left( \sum_{v_1 \in V_1} w_{v_1 \rightarrow v_2} a_{v_1}(x, \theta)\left( \sum_{v_0 \in V_0}  w_{v_0 \rightarrow v_1} x_{v_0} + b_{v_1} \right) \right) + b_{v_2} \\
& = \left(\sum_{\substack{v_0 \in V_0 \\ v_1 \in V_1}} w_{v_1 \rightarrow v_2} a_{v_1}(x, \theta) w_{v_0 \rightarrow v_1} x_{v_0} \right) + \left( \sum_{v_1 \in V_1}w_{v_1 \rightarrow v_2} a_{v_1}(x, \theta) b_{v_1} \right) + b_{v_2} \\
& = \left(\sum_{\substack{v_0 \in V_0 \\ v_1 \in V_1}}x_{v_0} w_{v_0 \rightarrow v_1} a_{v_1}(x, \theta)w_{v_1 \rightarrow v_2} \right)  + \left( \sum_{v_1 \in V_1}b_{v_1} a_{v_1}(x, \theta)w_{v_1 \rightarrow v_2}  \right) 
 + b_{v_2} \\
\end{aligned}
\end{equation*}
which proves \eqref{expression of f}, when $L=2$.

Now let $L \geq 3$ and suppose \eqref{expression of f} holds for all ReLU networks with $L-1$ layers. Let us consider a network with $L$ layers.

Let us denote by $g_{\theta}(x)$ the output of the $L-1$ first layers of the network pre-activation (before applying the ReLUs of the layer $L-1$). The function $g_\theta$ is that of a ReLU network with $L-1$ layers, and we have
\begin{equation*}
\begin{aligned}
f_{\theta}(x) & = W_L \sigma(g_{\theta}(x)) + b_L. 
\end{aligned}
\end{equation*}
Let $v_L \in V_L$. We thus have
\begin{equation}\label{f_theta in function of g_theta}
\begin{aligned}
f_{\theta}(x)_{v_L} & = \sum_{v_{L-1} \in V_{L-1}} w_{v_{L-1} \rightarrow v_L} \sigma(g_{\theta}(x)_{v_{L-1}})  + b_{v_L}.
\end{aligned}
\end{equation}

By the induction hypothesis, for all $v_{L-1} \in V_{L-1}$, $g_{\theta}(x)_{v_{L-1}}$ can be expressed with \eqref{expression of f}. Considering that $\sigma(g_{\theta}(x)_{v_{L-1}}) = a_{v_{L-1}}(x, \theta)g_{\theta}(x)_{v_{L-1}}$ and replacing $g_{\theta}(x)_{v_{L-1}}$ by its expression using \eqref{expression of f}, \eqref{f_theta in function of g_theta} becomes
\begin{align*}
f_{\theta}(x)_{v_L} & = \sum_{v_{L-1} \in V_{L-1}} w_{v_{L-1} \rightarrow v_L} a_{v_{L-1}}(x, \theta) \Bigg[ \Bigg(\sum_{\substack{v_{0} \in V_{0} \\ \vdots \\ v_{L-2} \in V_{L-2}}}   x_{v_0} w_{v_0 \rightarrow v_1} \prod_{l = 1}^{L-2} a_{v_l}(x, \theta) w_{v_{l} \rightarrow v_{l+1}} \Bigg)  \\
 & \qquad  + \Bigg( \sum_{l=1}^{L-2} \sum_{\substack{v_{l} \in V_{l} \\ \vdots \\ v_{L-2} \in V_{L-2}}}   b_{v_l} \prod_{l' = l}^{L-2} a_{v_{l'}}(x, \theta) w_{v_{l'} \rightarrow v_{l'+1}} \Bigg) + b_{v_{L-1}}\Bigg]  + b_{v_L} \\
& = \Bigg( \sum_{\substack{v_{0} \in V_{0} \\ \vdots \\ v_{L-1} \in V_{L-1}}} w_{v_{L-1} \rightarrow v_L} a_{v_{L-1}}(x, \theta)  x_{v_0} w_{v_0 \rightarrow v_1} \prod_{l = 1}^{L-2} a_{v_l}(x, \theta) w_{v_{l} \rightarrow v_{l+1}} \Bigg)  \\
 & \qquad  + \Bigg( \sum_{l=1}^{L-2} \sum_{\substack{v_{l} \in V_{l} \\ \vdots \\ v_{L-1} \in V_{L-1}}} w_{v_{L-1} \rightarrow v_L} a_{v_{L-1}}(x, \theta)  b_{v_l} \prod_{l' = l}^{L-2} a_{v_{l'}}(x, \theta) w_{v_{l'} \rightarrow v_{l'+1}} \Bigg) \\ 
 & \qquad+ \Bigg( \sum_{v_{L-1} \in V_{L-1}} w_{v_{L-1} \rightarrow v_L} a_{v_{L-1}}(x, \theta) b_{v_{L-1}} \Bigg) + b_{v_L} \\
&  = \Bigg(\sum_{\substack{v_{0} \in V_{0} \\ \vdots \\ v_{L-1} \in V_{L-1}}}   x_{v_0} w_{v_0 \rightarrow v_1} \prod_{l = 1}^{L-1} a_{v_l}(x, \theta) w_{v_{l} \rightarrow v_{l+1}}  \Bigg)\\ &  \qquad + \Bigg(\sum_{l=1}^{L-1} \sum_{\substack{v_{l} \in V_{l} \\ \vdots \\ v_{L-1} \in V_{L-1}}}  b_{v_l} \prod_{l' = l}^{L-1} a_{v_{l'}}(x, \theta) w_{v_{l'} \rightarrow v_{l'+1}}   \Bigg) + b_{v_L}, \\
\end{align*}
which proves \eqref{expression of f} holds for ReLU networks with $L$ layers. This ends the induction, and we conclude that \eqref{expression of f} holds for all ReLU networks.

We can now use this expression to prove Proposition \ref{fundamental prop of alpha}. The first sum in \eqref{expression of f} is taken over all the paths $p = (v_0, \dots, v_{L-1}) \in \mathcal{P}_{0}$, and each summand can be written as
\[ x_{v_0} w_{v_0 \rightarrow v_1} \prod_{l = 1}^{L-1} a_{v_l}(x, \theta) w_{v_{l} \rightarrow v_{l+1}} = \left( x_{v_0} \prod_{l=1}^{L-1} a_{v_l}(x, \theta)  \right) \left( \prod_{l = 0}^{L-1}  w_{v_{l} \rightarrow v_{l+1}}\right) = \alpha_p(x, \theta) \phi_{p, v_L}(\theta).\]
For all $l \in \llbracket 1, L-1 \rrbracket$, the inner sum of the double sum in \eqref{expression of f} is taken over all the paths $p=(v_l, \dots, v_{L-1}) \in \mathcal{P}_l $, and each summand can be written as
\[  b_{v_l} \prod_{l' = l}^{L-1} a_{v_{l'}}(x, \theta) w_{v_{l'} \rightarrow v_{l'+1}}  = \left( \prod_{l'=l}^{L-1} a_{v_{l'}}(x, \theta)  \right) \left(b_{v_{l}}\prod_{l' = l}^{L-1}  w_{v_{l'} \rightarrow v_{l'+1}} \right) = \alpha_p(x, \theta) \phi_{p, v_L}(\theta).\]
And finally, we can also write
\[b_{v_L} = \alpha_\beta (x, \theta) \phi_{\beta, v_L}(\theta). \]

Joining all these sums and denoting $\phi_{: , v_L}(\theta) = (\phi_{p, v_L}(\theta))_{p \in \mathcal{P}}  \in \RR^{\mathcal{P}}$, we have
\[f_{\theta}(x)_{v_L} = \sum_{p \in \mathcal{P}} \alpha_p(x, \theta) \phi_{p, v_L}(\theta) = \alpha(x, \theta) \phi_{: , v_L}(\theta),\]
so in other words,
\[f_{\theta}(x)^T =  \alpha(x, \theta) \phi(\theta).\]

\end{proof}

We restate here and prove Proposition \ref{alpha_X is piecewise-constant-main}.

\begin{prop}\label{alpha_X is piecewise-constant}
For all $n \in \NN^*$, for all $X \in \RR^{n \times V_0}$, the mapping
\[\fonction{\alpha_X}{\Par}{\RR^{n \times \mathcal{P}}}{\theta}{\alpha(X, \theta)}\]
appearing in \eqref{fundamental prop of alpha-matrix version-main} is piecewise-constant, with a finite number of pieces.
Furthermore, the boundary of each piece has Lebesgue measure zero. We call $\Delta_X$ the union of all the boundaries. The set $\Delta_X$ is closed and has Lebesgue measure zero.
\end{prop}

\begin{proof}
Let us first notice that for any $i \in \lb 1, n \rb$, for any $l \in \lb 1, L-1 \rb$,
\[\left( a_v(x^i,\theta)\right)_{v \in V_1 \cup \dots \cup V_{l-1}} \in \{ 0, 1 \}^{V_1 \cup \dots \cup V_{l-1}} \]
takes at most $2^{N_1 + \dots + N_{l-1}}$ distinct values, so the mapping $ \theta \mapsto \left( a_v(x^i,\theta)\right)_{v \in V_1 \cup \dots \cup V_{l-1}} $ is piecewise constant, with a finite number of pieces.

Let $i \in \lb 1, n \rb$. Let $l \in \lb 1, L-1 \rb$ and $v \in V_l$. Recall the definition of $f_{l-1}$, as given in Section \ref{ReLU networks-sec-main}. The function $\theta \rightarrow a_v(x^i, \theta)$ takes only two values, $1$ or $0$, and its values are determined by the sign of 
\begin{equation}\label{polynomial equation relu}
\sum_{v' \in V_{l-1}} w_{v' \rightarrow v} f_{l-1}(x^i)_{v'} + b_v.
\end{equation}
For all $v' \in V_{l-1}$, the value of $f_{l-1}(x^i)_{v'}$ depends on $\theta$. On a piece $P \subset \Par $ such that $\left(a_{v''}(x^i, \theta)\right)_{v'' \in V_1 \cup \dots \cup V_{l-1}}$ is constant, this dependence is polynomial. Thus, on $P$, the value of \eqref{polynomial equation relu} is a polynomial function of $\theta$, and since the coefficient applied to $b_v$ is equal to $1$, the corresponding polynomial is non constant. 
Since the values of $a_v(x^i, \theta)$ are determined by the sign of \eqref{polynomial equation relu}, inside $P$, the boundary between $\{ \theta \in \Par, a_v(x^i, \theta) = 0 \}$ and $\{ \theta \in \Par, a_v(x^i, \theta) = 1\}$ is included in the set of $\theta$ for which  \eqref{polynomial equation relu} equals $0$. This piece of boundary is thus contained in a level set of a non constant polynomial, whose Lebesgue measure is zero.

Since there is a finite number of pieces $P$, the Lebesgue measure of the boundary between $\{ \theta \in \Par, a_v(x^i, \theta) = 0 \}$ and $\{ \theta \in \Par, a_v(x^i, \theta) = 1\}$, which is contained in the union of the boundaries on all the pieces $P$, is thus equal to $0$.

Since this is true for all $l \in \lb 1, L-1 \rb$ and all $v \in V_l$, the boundary of a piece over which 
$\left( a_v(x^i,\theta)\right)_{v \in V_1 \cup \dots \cup V_{L-1}}$
is constant also has Lebesgue measure zero.

Now since, for all $x^i$, the value of $\alpha(x^i, \theta)$ only depends on 
$\left( a_v(x^i,\theta)\right)_{v \in V_1 \cup \dots \cup V_{L-1}}  $
and since $\alpha_X(\theta)$ is a matrix whose lines are the vectors $\alpha(x^i, \theta)$, we can conclude that $\fonction{\alpha_X}{\Par}{\RR^{n \times \mathcal{P}}}{\theta}{\alpha(X, \theta)}$
is piecewise-constant, with a finite number of pieces, and that the boundary of each piece has Lebesgue measure zero. 

A boundary is, by definition, closed. Finally, a finite union of closed sets with Lebesgue measure $0$, as $\Delta_X$ is, is closed and has Lebesgue measure $0$.
\end{proof}

For convenience, we introduce the two following notations. Let $l \in \lb 0, L \rb$. For any $l' \in \lb 0, l \rb$ and any path $p_i = (v_{l'}, \dots, v_l) \in V_{l'} \times \dots \times V_l$, we denote
\begin{equation}\label{definition p_i}
    \theta_{p_i} = \begin{cases}\prod_{k=0}^{l-1} w_{v_{k} \rightarrow v_{k+1}}& \text{if } l' =0 \\
b_{l'}\prod_{k=l'}^{l-1} w_{v_{k} \rightarrow v_{k+1}}& \text{if } l' \geq 1, \end{cases}
\end{equation}
where as a classic convention, an empty product is equal to $1$. In particular, if $l=0$, for any $p_i = (v_0) \in V_0$, we have $\theta_{p_i} = 1$.
For any path $p_o = (v_l, \dots, v_L) \in V_l \times \dots \times V_L$, we denote
\begin{equation}\label{definition p_o}
    \theta_{p_o} = \prod_{k=l}^{L-1} w_{v_{k} \rightarrow v_{k+1}},
\end{equation}
with again the convention that an empty product is equal to $1$, so if $l=L$, $\theta_{p_o} = 1$. 

Some attention must be paid to the fact that for any $l' \in \lb 1, L \rb$, if we take $p_i$ in the case $l=L$ and $p_o$ in the case $l=l'$, it is possible to have 
\[p_i = (v_{l'}, \dots, v_L) = p_o,\]
but in that case we DO NOT have $\theta_{p_i} = \theta_{p_o}$,
since $\theta_{p_i} = b_{l'}\prod_{k=l'}^{L-1} w_{v_{k} \rightarrow v_{k+1}}$ and $\theta_{p_o} = \prod_{k=l'}^{L-1} w_{v_{k} \rightarrow v_{k+1}}$. We will always denote the paths $p_i$ and $p_o$ with an $i$ (as in `input') or an $o$ (as in `output') to clarify which definition is used.

When considering another parameterization $\tilde{\theta} \in \Par$, we denote by $\tilde{\theta}_{p_i}$ and $\tilde{\theta}_{p_o}$ the corresponding objects. 

We establish different characterizations of the set $S$ defined in Section \ref{Rescaling-sec-main} that will be useful in the proofs. As mentioned in Section \ref{Rescaling-sec-main}, the subset of parameters $\ParS$ is close to the notion of `admissible' parameter in \cite{stock:hal-03292203}, but is slightly larger since the condition $w_{\bullet \rightarrow v} \neq 0$ is replaced by $(w_{\bullet \rightarrow v}, b_v) \neq (0,0)$, for each hidden neuron $v$.

\begin{prop}\label{nonzero input and output paths}
Let $\theta \in \Par$. The following statements are equivalent.

\item[i)] $\theta \in \ParS$.

\item[ii)] For all  $l \in \lb 1 , L-1 \rb$ and all $v_l \in V_l$, there exist $l' \in \lb 0, l \rb$, a path $p_i = (v_{l'}, \dots, v_l) \in V_{l'} \times \dots \times V_l$ and a path $p_o=(v_l, \dots, v_L) \in V_l \times \dots \times V_L$ such that
\[\theta_{p_i} \neq 0 \quad \text{and} \quad \theta_{p_o} \neq 0.\]

\item[iii)] For all  $l \in \lb 1 , L-1 \rb$ and all $v_l \in V_l$, there exist $l' \in \lb 0, l \rb$, a path $p = (v_{l'}, \dots, v_l, \dots, v_{L-1}) \in \mathcal{P}_{l'}$ and $ v_L \in V_L$ such that
\[\phi_{p,v_L}(\theta) \neq 0.\]
\end{prop}

\begin{proof} Let us show successively that $i) \Rightarrow ii)$, $ii) \Rightarrow iii)$ and $iii) \Rightarrow i)$.
\item[$i) \rightarrow ii)$] Let $\theta \in \ParS$. Let us show $ii)$ holds.

Let $l \in \lb 1 , L \rb$ and $v_l \in V_l$. To form a path $p_i$ satisfying the condition, we follow the procedure:
\begin{algorithmic}
\STATE $p_i \gets (v_l)$
\STATE $k \gets l$
\WHILE {$k \geq 1$ and $b_k =0$} 
  \STATE $\exists v_{k-1} \in V_{k-1}, w_{v_{k-1} \rightarrow v_k} \neq 0$
  \STATE $p_i \gets (v_{k-1},p_i)$
  \STATE $k \gets k-1$
\ENDWHILE
\STATE $l' \gets k$
\end{algorithmic}
The existence of $v_{k-1}$ in the loop is guaranteed by the fact that $\theta \not\in S$ and $b_k=0$ in the condition of the while loop. In the end, we obtain a path $p_i = (v_{l'}, \dots, v_l)$ with either $l' > 0$ and $b_{l'} \neq 0$, or $l' =0$. In both cases, we have by construction
\[\theta_{p_i} \neq 0.\]

We do similarly the other way to form a path $p_o = (v_l, \dots, v_L)$. We follow the procedure:

\begin{algorithmic}
\STATE $p_o \gets (v_l)$
\STATE $k \gets l$
\WHILE {$k \leq L-1$} 
  \STATE $\exists v_{k+1} \in V_{k+1}, w_{v_{k} \rightarrow v_{k+1}} \neq 0$
  \STATE $p_o \gets (p_o, v_{k+1})$
  \STATE $k \gets k+1$
\ENDWHILE
\end{algorithmic}
The existence of $v_{k+1}$ in the loop is guaranteed by the fact that $\theta \not\in S$. In the end, we obtain a path $p_o = (v_l, \dots, v_L)$ satisfying by construction
\[\theta_{p_o} \neq 0.\]

\item[$ii) \rightarrow iii)$] Let $l\in \lb 1, L-1 \rb$ and $v_l \in V_l$. There exist $l' \in \lb 0, l \rb$, a path $p_i = (v_{l'}, \dots, v_l) \in V_{l'} \times  \dots \times V_l$ and a path $p_o=(v_l, \dots, v_L) \in V_l \times \dots \times V_L$ such that
\[\theta_{p_i} \neq 0 \quad \text{and} \quad \theta_{p_o} \neq 0.\]
Denoting $p=(v_{l'}, \dots, v_l, \dots, v_{L-1})$, we have
\[\phi_{p, v_L}(\theta) = \theta_{p_i} \theta_{p_o} \neq 0.\]

\item[$iii) \rightarrow i)$] Let us show the contrapositive: let $\theta \in S$, and let us show the statement $iii)$ is not true. Indeed, if $\theta \in S$, there exist $l \in \lb 1, L-1 \rb$ and $v_l \in V_l$ such that $(w_{\bullet \rightarrow v_l}, b_{v_l})=(0,0)$ or $w_{v_l \rightarrow \bullet} = 0$. Consider a path $p=(v_{l'}, \dots, v_l, \dots, v_{L-1})$ and $v_L \in V_L$. We have
\[\phi_{p,v_L}(\theta) = \begin{cases}b_{v_{l'}}w_{v_{l'} \rightarrow v_{l'+1}} \dots w_{v_{l-1} \rightarrow v_{l}} w_{v_{l} \rightarrow v_{l+1}} \dots w_{v_{L-1} \rightarrow v_L} & \text{if } l' \geq 1 \\ 
w_{v_{0} \rightarrow v_{1}} \dots w_{v_{l-1} \rightarrow v_{l}} w_{v_{l} \rightarrow v_{l+1}} \dots w_{v_{L-1} \rightarrow v_L} & \text{if } l'=0. \end{cases}\]
If $(w_{\bullet \rightarrow v_l}, b_{v_l})=(0,0)$, either $l'=l$ and $b_{v_{l'}} = 0$ so $\phi_{p,v_L}(\theta) = 0$, or $l' < l$ and since $w_{v_{l-1} \rightarrow v_l} = 0$, we have $\phi_{p,v_L}(\theta) = 0$. 

If $w_{v_l \rightarrow \bullet} = 0$, $w_{v_l \rightarrow v_{l+1}} = 0$ so $\phi_{p,v_L}(\theta) = 0$.  Thus $iii)$ is not satisfied. 
\end{proof}

We restate and prove Proposition \ref{equivalence implies equal phi-main}.
\begin{prop}\label{equivalence implies equal phi}
For all $\theta, \tilde{\theta} \in \Par$, we have
\[\theta \overset{R}{\sim} \tilde{\theta} \quad \Longrightarrow \quad \phi(\theta) = \phi( \tilde{\theta}),\]
and thus in particular
\[\theta \sim \tilde{\theta} \quad \Longrightarrow \quad \phi(\theta) = \phi(\tilde{\theta}).\]
\end{prop}

\begin{proof}
Let $\theta, \tilde{\theta} \in \Par$ such that $\theta \overset{R}{\sim} \tilde{\theta}$. There exists a family $(\lambda^0, \dots, \lambda^L) \in (\RR^*)^{V_0} \times \dots \times (\RR^*)^{V_L}$, with $\lambda^0 = \One_{V_0}$ and $\lambda^L = \One_{V_L}$, such that for all $l \in \lb 1 , L \rb$, for all $(v_{l-1}, v_l) \in V_{l-1} \times V_l$, \eqref{equivalence def-2-main} holds.
We consider first a path $\p = (v_0, \dots, v_{L-1}) \in \mathcal{P}_0$ and $v_L \in V_L$. Using \eqref{equivalence def-2-main} and the fact that $\lambda^{0}_{v_{0}} = \lambda^L_{v_L} =1$, we have 
\[\phi_{p,v_L}(\theta) =\prod_{l=1}^{L} w_{v_{l-1} \rightarrow v_{l}} = \prod_{l=1}^{L} \frac{\lambda^l_{v_l}}{\lambda^{l-1}_{v_{l-1}}} \tilde{w}_{v_{l-1} \rightarrow v_l} = \frac{\lambda^L_{v_L}}{\lambda^{0}_{v_{0}}} \prod_{l=1}^{L}  \tilde{w}_{v_{l-1} \rightarrow v_l} = \phi_{p,v_L}(\tilde{\theta}).\]
Similarly, for $l \in \llbracket 1, L-1 \rrbracket$ and a path $p =(v_l, \dots, v_{L-1}) \in \mathcal{P}_l$, and for all $v_L \in V_L$, we have, using \eqref{equivalence def-2-main} and the fact that $\lambda^L_{v_L}=1$,
\begin{equation*}
\begin{aligned}
\phi_{p,v_L}(\theta) & = b_{v_l} \prod_{l'=l+1}^{L} w_{v_{l'-1} \rightarrow v_{l'}} = \lambda^l_{v_l} \tilde{b}_{v_l} \prod_{l'=l+1}^{L}  \frac{\lambda^{l'}_{v_{l'}}}{\lambda^{l'-1}_{v_{l'-1}}} \tilde{w}_{v_{l'-1} \rightarrow v_{l'}}  = \lambda^L_{v_L} \tilde{b}_{v_l} \prod_{l'=l+1}^{L}  \tilde{w}_{v_{l'-1} \rightarrow v_{l'}} \\
& = \phi_{p,v_L}(\tilde{\theta}).
\end{aligned}
\end{equation*}
Finally, for $p = \beta$ and $v_L \in V_L$, we have
\[\phi_{p,v_L}(\theta) = b_{v_L} = \lambda^L_{v_L} \tilde{b}_{v_L} = \tilde{b}_{v_L} =  \phi_{p,v_L}(\tilde{\theta}).\]

This shows $\phi(\theta) = \phi(\tilde{\theta})$.

For the second implication, we simply use the fact that if $\theta \sim \tilde{\theta}$, in particular, $\theta \overset{R}{\sim} \tilde{\theta}$.
\end{proof}

\begin{coro}\label{ParS is stable by r equivalence}
The set $\ParS$ is stable by rescaling equivalence: if $\theta \in \ParS$, and $\tilde{\theta} \in \Par$ satisfies $\theta \overset{R}{\sim} \tilde{\theta}$, then $\tilde{\theta} \in \ParS$.
\end{coro}

\begin{proof}
Let $\theta \in \ParS$ and $\tilde{\theta} \in \Par$ such that $\theta \overset{R}{\sim} \tilde{\theta}$. Proposition \ref{equivalence implies equal phi} shows that $\phi(\tilde{\theta}) = \phi(\theta)$.

Let $l \in \lb 1 , L \rb$ and $v \in V_l$. Since $\theta \in \ParS$, according to Proposition \ref{nonzero input and output paths} there exists $l' \in \lb 0, l \rb$, a path $p = (v_{l'}, \dots, v_l, \dots , v_{L-1}) $ and $v_L \in V_L$ such that $\phi_{p,v_L}(\theta)  \neq 0$. We have
\[  \phi_{p,v_L}(\tilde{\theta}) = \phi_{p,v_L}(\theta) \neq 0,\]
and since this is true for any $l \in \lb 1 , L \rb$ and $v \in V_l$, Proposition \ref{nonzero input and output paths} shows that $\tilde{\theta} \in \ParS$.
\end{proof}

We restate and prove Proposition \ref{equal phi implies resc equivalence-main}.
\begin{prop}\label{equal phi implies resc equivalence}
For all $\theta \in \ParS $, for all $\tilde{\theta} \in \Par$, 
\[\phi(\theta) = \phi( \tilde{\theta}) \quad \Longrightarrow \quad \theta \overset{R}{\sim} \tilde{\theta}.\]
\end{prop}

\begin{proof}
Let us choose $(\lambda^0, \dots, \lambda^L) \in (\RR^*)^{V_0} \times \dots \times (\RR^*)^{V_L} $ as follows. For all $l \in \lb 1, L-1 \rb$ and all $v_l \in V_l$, since $\theta \in \ParS$, Proposition \ref{nonzero input and output paths} shows that there exists a path $p_o(v_l) = (v_{l}, \dots, v_L) \in V_{l} \times \dots \times V_L$ such that $\theta_{p_o(v_l)} \neq 0$. Let us define $\lambda^0 = \One_{V_0}$, $\lambda^L = \One_{V_L}$ and for all $l \in \lb 1, L-1 \rb$, 
\[\lambda^l_{v_l} = \frac{\tilde{\theta}_{p_o(v_l)}}{\theta_{p_o(v_l)}}.\]
The value of $\lambda^l_{v_l}$ a priori depends on the choice of the path $p_o(v_l)$, but the first of the two following facts, that we are going to prove, shows it only depends on $v_l$, since in \eqref{theta_pi = lambda_vl tilde(theta)_pi}, $p_i$ does not depend on $p_o(v_l)$. 
\begin{itemize}
\item For all $l \in \lb 0, L \rb$, for all $v_l \in V_l$, for any $l' \in \lb 0, l \rb$ and any $p_i =(v_{l'}, \dots, v_l) \in V_{l'} \times \dots \times V_l$,
\begin{equation}\label{theta_pi = lambda_vl tilde(theta)_pi}
\theta_{p_i} = \lambda^l_{v_l} \tilde{\theta}_{p_i}.
\end{equation}
\item For all $l \in \lb 0, L \rb$, for all $v_l \in V_l$,  $\lambda^l_{v_l} \neq 0$.
\end{itemize}
Indeed, let $l \in \lb 0 , L \rb$ and let us consider $l' \in \lb 0, l \rb$ and a path $p_i = (v_{l'}, \dots, v_l) \in V_{l'} \times \dots \times V_l$. Let $v_{l+1}, \dots, v_L \in V_{l+1} \times \dots \times V_L$ such that $p_o(v_l) = (v_l, v_{l+1}, \dots, v_L)$. Let $p =(v_{l'}, \dots, v_l, \dots, v_{L-1})\in \mathcal{P}_{l'}$ so that $p_i \cup p_o(v_l) = p \cup (v_L)$. We have by hypothesis
\[ \theta_{p_i} \theta_{p_o(v_l)} = \phi_{p, v_L}(\theta) = \phi_{p, v_L}(\tilde{\theta}) = \tilde{\theta}_{p_i} \tilde{\theta}_{p_o(v_l)},\]
thus
\[\theta_{p_i} = \frac{\tilde{\theta}_{p_o(v_l)}}{\theta_{p_o(v_l)}} \tilde{\theta}_{p_i} = \lambda^l_{v_l} \tilde{\theta}_{p_i},\]
which proves the first point. To prove the second point, we simply use Proposition \ref{nonzero input and output paths} to consider a path $p_i$ such that $\theta_{p_i} \neq 0$, and \eqref{theta_pi = lambda_vl tilde(theta)_pi} shows that $\lambda^l_{v_l} \neq 0$.

Let us now prove the rescaling equivalence.
Let $l \in \lb 1, L \rb$, and let $(v_{l-1}, v_l) \in V_{l-1} \times V_l$. Let us consider, thanks to Proposition \ref{nonzero input and output paths},  $l' \in \lb 0 , l-1 \rb$ and a path $p_i = (v_{l'}, \dots, v_{l-1}) \in V_{l'} \times \dots \times V_{l-1}$ such that $\theta_{p_i} \neq 0$. The relation \eqref{theta_pi = lambda_vl tilde(theta)_pi} shows we also have $\tilde{\theta}_{p_i} \neq 0$. Let $p_i' = p_i \cup (v_l)$. Using \eqref{theta_pi = lambda_vl tilde(theta)_pi} with $\theta_{p_i'}$ we have
\[ \theta_{p_i} w_{v_{l-1} \rightarrow v_{l}} = \theta_{p_i'} = \lambda^l_{v_l} \tilde{\theta}_{p_i'} = \lambda^l_{v_l}  \tilde{\theta}_{p_i} \tilde{w}_{v_{l-1} \rightarrow v_l}.\]
At the same time, using \eqref{theta_pi = lambda_vl tilde(theta)_pi} with $\theta_{p_i}$ we have,
\[ \theta_{p_i} w_{v_{l-1} \rightarrow v_{l}} = \lambda^{l-1}_{v_{l-1}} \tilde{\theta}_{p_i} w_{v_{l-1} \rightarrow v_{l}},  \]
so combining both equalities, we have
\[\lambda^l_{v_l}  \tilde{\theta}_{p_i} \tilde{w}_{v_{l-1} \rightarrow v_l} = \lambda^{l-1}_{v_{l-1}} \tilde{\theta}_{p_i} w_{v_{l-1} \rightarrow v_{l}} .\]
Using the fact that $\tilde{\theta}_{p_i} \neq 0$ and $\lambda^{l-1}_{v_{l-1}} \neq 0$, we finally obtain, for all $l \in \lb 1, L \rb$ and all $(v_{l-1}, v_l) \in V_{l-1} \times V_l$:
\[w_{v_{l-1} \rightarrow v_{l}} = \frac{\lambda^l_{v_l}}{\lambda^{l-1}_{v_{l-1}}} \tilde{w}_{v_{l-1} \rightarrow v_l}.\]
For all $l \in \lb 1, L \rb$ and all $v_l \in V_l$, using \eqref{theta_pi = lambda_vl tilde(theta)_pi} with $p_i = (v_l)$, we obtain
\[b_{v_l} =  \lambda^l_{v_l} \tilde{b}_{v_l}.\]
This shows that \eqref{equivalence def-2-main} is satisfied for all $(v_{l-1},v_l) \in V_{l-1} \times V_l$, and thus $\theta \rsim \tilde{\theta}$.
\end{proof}

The following proposition is useful in the proof of Theorem \ref{local identifiability using Sigma1*} and allows to improve identifiability modulo rescaling into identifiability modulo positive rescaling. 
\begin{prop}\label{rescaling equivalent and close implies equivalent}
For all $\theta \in \ParS$, there exists $\epsilon > 0$ such that for all $\tilde{\theta} \in \Par$, 
\[\|\theta - \tilde{\theta}\|_{\infty} < \epsilon \text{ and } \theta \rsim \tilde{\theta} \quad \Longrightarrow \quad \theta \sim \tilde{\theta}.\]
\end{prop}

\begin{proof}
Let $\theta \in \ParS$. We define
\[\epsilon = \min \Big( \big\{ |w_{v \rightarrow v'} |, \ v \rightarrow v' \in E \text{ and } w_{v \rightarrow v'} \neq 0 \big\} \ \cup \ \big\{ |b_v |, \ v \in B \text{ and } b_v \neq 0 \big\} \Big).\]
Let $\tilde{\theta} \in \Par$ such that $\| \theta - \tilde{\theta}\|_\infty < \epsilon$ and $\theta \rsim \tilde{\theta}$. To prove $\theta \sim \tilde{\theta}$, we simply have to prove $\sgn(\theta) = \sgn(\tilde{\theta})$. There exists \mbox{$(\lambda^0, \dots, \lambda^L) \in (\RR^*)^{V_0} \times \dots \times (\RR^*)^{V_L}$}, with $\lambda^0 = \One_{V_0}$ and $\lambda^L = \One_{V_L}$, such that, for all $l \in \lb 1, L \rb$, for all $(v_{l-1}, v_l) \in V_{l-1} \times V_l$, \eqref{equivalence def-2-main} holds. Let us show that $\sgn (\theta) = \sgn(\tilde{\theta})$. 

Indeed, let $l \in \lb 1, L \rb$, and let $(v, v') \in V_{l-1} \times V_l$. If $w_{v \rightarrow v'} \neq 0$, then since \mbox{$|w_{v \rightarrow v'} - \tilde{w}_{v \rightarrow v'} | < \epsilon$} and by definition $\epsilon \leq |w_{v \rightarrow v'}|$, we have $\sgn(w_{v \rightarrow v'})= \sgn(\tilde{w}_{v \rightarrow v'})$. Otherwise, if $w_{v \rightarrow v'} = 0$, \eqref{equivalence def-2-main} shows that we have
\[\tilde{w}_{v \rightarrow v'} = \frac{\lambda^{l-1}_{v}}{\lambda^l_{v'}}w_{v \rightarrow v'} = 0,\]
so we still have $\sgn(w_{v \rightarrow v'})=\sgn(\tilde{w}_{v \rightarrow v'})$. 

Now let $l \in \lb 1, L \rb$ and let $v \in V_l$. Similarly, if $b_v \neq 0$, we have $|b_v - \tilde{b}_v| < \epsilon \leq |b_v|$, so $\sgn(b_v)= \sgn(\tilde{b}_v)$, and if $b_v = 0$, we have
\[\tilde{b}_v = \frac{b_v}{\lambda^l_v} = 0,\]
so again $\sgn(b_v)= \sgn(\tilde{b}_v)$. 

This shows $\sgn(\theta) = \sgn(\tilde{\theta})$, so $\theta \sim \tilde{\theta}$.
\end{proof}

\section{The smooth manifold structure of \texorpdfstring{$\Sigma_1^*$}{}}\label{The smooth manifold structure of Sigma_1^*-sec}

In this section, we prove Theorem \ref{Sigma_1^* is a smooth manifold-main}, which is restated as Theorem \ref{Sigma_1^* is a smooth manifold}. Before doing so, we establish intermediary results, some of which are evoked in Section \ref{The smooth manifold Sigma_1^*-sec-main}.

Let us discuss the cardinal of $F_\theta$ defined in Section \ref{The smooth manifold Sigma_1^*-sec-main}. The set $F_\theta$ is obtained by removing the edges of the form $v \rightarrow s^\theta_{\max}(v)$ for $v \in V_1 \cup \dots \cup V_{L-1}$. Note that we do not remove the edges of the form $v \rightarrow s^\theta_{\max}(v)$ for $v \in V_0$.
For all $l \in \lb 1, L-1 \rb$, there are precisely $N_l$ edges of the form $(v, s^\theta_{\max}(v))$ with $v \in V_l$, so
\begin{equation*}
\begin{aligned}
|F_\theta| & = |E| - \left( N_1 + \dots + N_{L-1} \right) \\
& = N_0N_1 + \dots + N_{L-1}N_L - N_1 - \dots - N_{L-1}.
\end{aligned}
\end{equation*}

As a consequence, since $|B| = N_1 + \dots + N_L$, we have in particular
\begin{align}\label{dimension of Sigma_1^*}
    |F_\theta| + |B| & = N_0N_1 + \dots + N_{L-1}N_L - N_1 - \dots - N_{L-1} + N_1 + \dots + N_L \nonumber \\
    & = N_0N_1 + \dots + N_{L-1}N_L + N_L.
\end{align}

The following proposition is a first step towards Proposition \ref{psi^theta is a homeomorphism}, which states that $\psi^\theta$ is a homeomorphism.
\begin{prop}\label{psi^theta is injective}
For all $\theta \in \ParS$, the function $\psi^{\theta}: U_\theta \rightarrow \RR^{\mathcal{P}\times V_L}$ is injective.
\end{prop}

\begin{proof}
Let $\tau, \tilde{\tau} \in U_\theta$ such that $\psi^\theta(\tau) = \psi^\theta(\tilde{\tau})$. Let us show $\tau = \tilde{\tau}$.
We have $\phi(\rho_\theta(\tau)) = \phi(\rho_\theta(\tilde{\tau}))$ and by definition of $U_\theta$, $\rho_\theta(\tau) \in \ParS$, so by Proposition \ref{equal phi implies resc equivalence} we have the rescaling equivalence
\[\rho_\theta(\tau) \overset{R}{\sim} \rho_\theta(\tilde{\tau}).\]
By definition of the rescaling equivalence, in its formulation \eqref{equivalence def-2-main}, there exists $(\lambda^0, \dots, \lambda^L) \in (\RR^*)^{V_0} \times \dots \times (\RR^*)^{V_L}$, with $\lambda^0 = \One_{V_0}$ and $\lambda^L = \One_{V_L}$, such that, for all $l \in \llbracket 1 , L \rrbracket$, for all $(v_{l-1}, v_l) \in V_{l-1} \times V_l$,
\begin{equation}\label{proof injectivity psi-1}
\begin{cases}\rho_\theta(\tau)_{v_{l-1} \rightarrow v_l} = \frac{(\lambda^l)_{v_l}}{(\lambda^{l-1})_{v_{l-1}}}\rho_\theta(\tilde{\tau})_{v_{l-1} \rightarrow v_l} \\
b_{v_l} = \lambda^l_{v_l} \tilde{b}_{v_l}. \end{cases}
\end{equation}

Let $l \in \llbracket 2, L \rrbracket$ and let $v_{l-1} \in V_{l-1}$. Let $v_l = s^\theta_{\max}(v_{l-1})$. According to \eqref{proof injectivity psi-1} we have
\[\rho_\theta(\tau)_{v_{l-1} \rightarrow v_l} = \frac{(\lambda^l)_{v_l}}{(\lambda^{l-1})_{v_{l-1}}}\rho_\theta(\tilde{\tau})_{v_{l-1} \rightarrow v_l}.\]
But since $v_l = s^\theta_{\max}(v_{l-1})$ and $v_{l-1} \in V_{l-1}$ with $l-1 \in \lb 1, L-1 \rb$, we have $v_{l-1} \rightarrow v_l \in E\backslash F_\theta$, so by definition of $\rho_\theta$ in  \eqref{definition of rho_theta-main},
\[\rho_\theta(\tau)_{v_{l-1}\rightarrow v_l} = w_{v_{l-1} \rightarrow v_l} =  \rho_\theta(\tilde{\tau})_{v_{l-1} \rightarrow v_l} \neq 0, \]
so $\frac{(\lambda^l)_{v_l}}{(\lambda^{l-1})_{v_{l-1}}} = 1$.

We have shown that for all $l \in \lb 2 , L \rb $, for all $v_{l-1} \in V_{l-1}$, there exists $v_l \in V_l$ such that 
\begin{equation*}
(\lambda^{l-1})_{v_{l-1}}= (\lambda^l)_{v_l}.
\end{equation*}
As a consequence, if $l$ is such that $\lambda^l = \One_{V_l}$, then $\lambda^{l-1} = \One_{V_{l-1}}$.

Starting from $\lambda^L = \One_{V_L}$, this shows by induction that for all $l \in \lb 1, L \rb$, 
\[\lambda^l = \One_{V_l}.\]
By hypothesis we also have $\lambda^0 = \One_{V_0}$. Using \eqref{proof injectivity psi-1}, this shows that
\[\rho_\theta(\tau) = \rho_\theta(\tilde{\tau}).\]
The injectivity of $\rho_\theta$ allows us to conclude that 
\[\tau = \tilde{\tau}.\]
\end{proof}

The following proposition shows, as mentioned in Section \ref{The smooth manifold Sigma_1^*-sec-main}, that $\psi^\theta$ is a homeomorphism. This is a necessary step to prove that $(V_\theta, (\psi^\theta)^{-1})_{\theta \in \ParS}$ is an atlas of $\Sigma_1^*$.
\begin{prop}\label{psi^theta is a homeomorphism}
For all $\theta \in \ParS$, $\psi^{\theta}$ is a homeomorphism from $U_{\theta}$ onto its image $V_\theta$.
\end{prop}

\begin{proof}We already know from Proposition \ref{psi^theta is injective} that $\psi^\theta$ is injective, so we need to prove that $\psi^\theta$ is continuous and its inverse is continuous. The function $\rho_\theta$ is affine and $\phi$ is a polynomial function, so the function $\psi^\theta = \phi \circ \rho_\theta$ is a polynomial function, and in particular it is continuous.

To prove that $(\psi^\theta)^{-1}$ is continuous, we consider a sequence $(\tau_n)$ taking values in $U_\theta$ and $\tau \in U_\theta$ such that $\psi^\theta(\tau_n) \rightarrow \psi^\theta(\tau)$, and we want to show that $\tau_n \rightarrow \tau$.

Let us first show that for all $v \in B$, $(\tau_n)_v \rightarrow \tau_v$. Indeed, let $l \in \lb 1, L \rb$ and let $v_l \in V_l$, so that $v_l$ is an arbitrary element of $B$. Let us define $v_{l+1} = s^\theta_{\max}(v_l)$, then $v_{l+2} = s^\theta_{\max}(v_{l+1})$ and so on up to $v_L = s^\theta_{\max}(v_{L-1})$. Since for all $l' \in \lb l, L-1 \rb$, $v_{l'+1} = s^\theta_{\max}(v_{l'})$, by definition of $F_\theta$ and $\rho_\theta$ (see \eqref{definition of F_theta-main} and \eqref{definition of rho_theta-main}), we have
\begin{equation}\label{proof of psi^theta is a homeo.-3}
    \rho_\theta(\tau_n)_{v_{l'} \rightarrow v_{l'+1}} = w_{v_{l'} \rightarrow v_{l'+1}},
\end{equation}
and
\begin{equation}\label{proof of psi^theta is a homeo.-4}
    \rho_\theta(\tau)_{v_{l'} \rightarrow v_{l'+1}} = w_{v_{l'} \rightarrow v_{l'+1}}.
\end{equation}
In particular, since $\theta \not\in S$, for all $l' \in \lb l, L-1 \rb$ we have $w_{v_{l'} \rightarrow \bullet} \neq 0$, so by definition of $s^\theta_{\max}$, $w_{v_{l'} \rightarrow v_{l'+1}} \neq 0$. We thus have
\begin{equation}\label{proof of psi^theta is a homeo.-1}
    w_{v_l \rightarrow v_{l+1}} \dots w_{v_{L-1} \rightarrow v_L} \neq 0.
\end{equation}
If we denote $p=(v_l, \dots, v_{L-1})$, we have, using the definition of $\phi$ and \eqref{proof of psi^theta is a homeo.-3},
\[\psi^\theta_{p,v_L}(\tau_n) = (\tau_n)_{v_l} w_{v_l \rightarrow v_{l+1}} \dots w_{v_{L-1} \rightarrow v_L}\]
and using \eqref{proof of psi^theta is a homeo.-4},
\[\psi^\theta_{p,v_L}(\tau) = (\tau)_{v_l} w_{v_l \rightarrow v_{l+1}} \dots w_{v_{L-1} \rightarrow v_L}.\]
Using \eqref{proof of psi^theta is a homeo.-1} and the fact that
\[\psi^\theta(\tau_n) \rightarrow \psi^\theta(\tau),\]
we conclude that
\[(\tau_n)_{v_l} \rightarrow \tau_{v_l}.\]

Let us now prove that for all $(v,v') \in E$, $(\tau_n)_{v \rightarrow v'} \rightarrow \tau_{v \rightarrow v'}$. Let us show by induction on $l \in \lb 1, L \rb$ the following hypothesis
\begin{equation}\tag{$H_l$}
\forall l' \in \lb 1, l \rb, \quad  \forall (v,v') \in \left( V_{l'-1} \times V_{l'} \right) \cap F_\theta, \quad (\tau_n)_{v \rightarrow v'} \longrightarrow \tau_{v \rightarrow v'}. 
\end{equation}

{\bf Initialization.} Let $(v_0,v_1) \in (V_0 \times V_1) \cap F_\theta$. We define $v_2 = s^\theta_{\max} (v_1)$, then we define $v_3 = s^\theta_{\max}(v_2)$, and so on up to $v_L = s^\theta_{\max}(v_{L-1})$. Let $p = (v_0, \dots, v_{L-1}) \in \mathcal{P}$.

As above, using the definition of $\rho_\theta$, $F_\theta$ and $\phi$, we have
\[\psi^\theta_{p,v_L}(\tau_n) = (\tau_n)_{v_0 \rightarrow v_1} w_{v_1 \rightarrow v_2} \dots w_{v_{L-1} \rightarrow v_L}\]
and
\[\psi^\theta_{p,v_L}(\tau) = (\tau)_{v_0 \rightarrow v_1} w_{v_1 \rightarrow v_2} \dots w_{v_{L-1} \rightarrow v_L},\]
and since $\theta \not\in S$, we also have , as above,
\begin{equation}\label{proof of psi^theta is a homeo.-2}
     w_{v_1 \rightarrow v_2} \dots w_{v_{L-1} \rightarrow v_L} \neq 0.
\end{equation}

Since 
\[\psi^\theta(\tau_n) \longrightarrow \psi^\theta(\tau)\]
we conclude using \eqref{proof of psi^theta is a homeo.-2} that 
\[ (\tau_n)_{v_0 \rightarrow v_1} \longrightarrow \tau_{v_0 \rightarrow v_1}.\]
We have shown $H_1$.
 
{\bf Induction step.} Let $l \in \lb 2, L \rb$ and let us assume that $H_{l-1}$ holds.
 
Let $(v_{l-1},v_l) \in \left( V_{l-1} \times V_l \right) \cap F_\theta$. We define $v_{l+1} = s^\theta_{\max}(v_l), v_{l+2} = s^\theta_{\max} (v_{l+1})$, and so on up to $v_L = s^\theta_{\max}(v_{L-1})$. Let us denote $p_o = (v_l, \dots, v_L)$. Recalling the notation defined in \eqref{definition p_o}, we have
\begin{equation}\label{rho(tau)_po = theta_po}
\rho_\theta(\tau_n)_{p_o} = w_{v_l \rightarrow v_{l+1}} \dots w_{v_{L-1} \rightarrow v_L} = \rho_\theta(\tau)_{p_o} \neq 0. 
\end{equation}
 
At the same time, since $\tau \in U_{\theta}$, Proposition \ref{nonzero input and output paths} shows there exist $l' \in \lb 0, l-1 \rb$ and a path $p_i = (v_{l'}, \dots, v_{l-2}, v_{l-1})$ such that
\begin{equation}\label{proof of psi^theta is a homeo.-5}
    \rho_\theta(\tau)_{p_i} \neq 0.
\end{equation}
If $l' \geq 1$, we have shown in the first part of the proof that $(\tau_n)_{v_{l'}} \longrightarrow \tau_{v_{l'}}$. Moreover, whatever the value of $l'$ is, for $k \in \lb l', l-2 \rb$, if $(v_k, v_{k+1}) \in E \backslash F_\theta$, 
\[\rho_\theta(\tau_n)_{v_k \rightarrow v_{k+1}} = w_{v_k \rightarrow v_{k+1}} = \rho_\theta(\tau)_{v_k \rightarrow v_{k+1}},\]
and if $(v_k, v_{k+1}) \in F_\theta$, according to $H_{l-1}$, 
\[\rho_\theta(\tau_n)_{v_k \rightarrow v_{k+1}} = (\tau_n)_{v_k \rightarrow v_{k+1}} \longrightarrow \tau_{v_k \rightarrow v_{k+1}} = \rho_\theta(\tau)_{v_k \rightarrow v_{k+1}}.\] We therefore have
\begin{equation}\label{rho(tau_n) -> rho(tau)}
\rho_\theta(\tau_n)_{p_i} \longrightarrow \rho_\theta(\tau)_{p_i},
\end{equation}
and in particular, since $\rho_\theta(\tau)_{p_i} \neq 0$, there exists $n_0 \in \NN$ such that for all $n \geq n_0$, 
\begin{equation}\label{rho(tau_n)_pi neq 0}
\rho_\theta(\tau_n)_{p_i} \neq 0.
\end{equation}
We can write
\[\psi^\theta_{p,v_L}(\tau_n) = \rho_\theta(\tau_n)_{p_i} \: (\tau_n)_{v_{l-1} \rightarrow v_l} \: \rho_\theta(\tau_n)_{p_o}\]
and
\[\psi^\theta_{p,v_L}(\tau) = \rho_\theta(\tau)_{p_i} \: (\tau)_{v_{l-1} \rightarrow v_l} \: \rho_\theta(\tau)_{p_o},\]
so using \eqref{rho(tau)_po = theta_po}, \eqref{rho(tau_n)_pi neq 0} and \eqref{rho(tau_n) -> rho(tau)}, we have
\[ (\tau_n)_{v_{l-1} \rightarrow v_l} = \frac{\psi^\theta_{p,v_L}(\tau_n) }{\rho_\theta(\tau_n)_{p_i} \rho_\theta(\tau_n)_{p_o}} \quad \longrightarrow \quad  \frac{\psi^\theta_{p,v_L}(\tau) }{\rho_\theta(\tau)_{p_i} \rho_\theta(\tau)_{p_o}} = \tau_{v_{l-1} \rightarrow v_l}.\]
We have shown $H_l$, which concludes the induction step.

In particular, $H_L$ is satisfied, and finally $\tau_n \rightarrow \tau$. 

This shows that $\psi^\theta$ is a homeomophism.
\end{proof}

The following lemma is necessary for the proof of Proposition \ref{psi^t(U_t) is a neighborhood of t}.
\begin{lem}\label{tilde(theta)_v->v' neq 0 close to theta}
Let $\theta \in \ParS$. Let $(v,v') \in E$ (resp. $v \in B$). If $w_{v \rightarrow v'} \neq 0$ (resp. $b_v \neq 0$), then there exists $\epsilon > 0$ such that for all $\tilde{\theta} \in \Par$, if $\| \phi(\theta) - \phi(\tilde{\theta}) \|_{\infty} < \epsilon$, then $\tilde{w}_{v \rightarrow v'} \neq 0$ (resp. $\tilde{b}_v \neq 0$).
\end{lem}

\begin{proof}
Let $\theta \in \ParS$ and $(v,v') \in E$ such that $w_{v \rightarrow v'} \neq 0$. Denote $l \in \lb 0,L-1 \rb$ such that $v \in V_l$. If $l=0$, we take $p_i = (v)$ so that by convention $\theta_{p_i} = 1 \neq 0$, and if $l \geq 1$, we use Proposition \ref{nonzero input and output paths} which states that there exists $l' \in \lb 0, l-1 \rb$ and a path $p_i = (v_{l'}, \dots, v_{l-2}, v)$ such that $\theta_{p_i} \neq 0$. Similarly, if $l = L-1$, we take $p_o = (v')$ so that by convention $\theta_{p_o} = 1 \neq 0$ and if $l < L-1$, we use Proposition \ref{nonzero input and output paths} which states that there exists a path $p_o = (v', v_{l+1}, \dots, v_L)$ such that $\theta_{p_o} \neq 0$. If we denote 
\[ p = \begin{cases} (v, v', v_{l+2}, \dots, v_{L-1}) & \text{if } l=0\\
(v_{l'}, \dots , v_{l-1}, v, v')  & \text{if } l=L-1\\
(v_{l'}, \dots , v_{l-1}, v, v', v_{l+2}, \dots, v_{L-1})& \text{otherwise,} 
\end{cases}\]
we have
\[\phi_{p,v_L}(\theta) = \theta_{p_i} w_{v \rightarrow v'} \theta_{p_o} \neq 0.\]
We define $\epsilon = |\phi_{p, v_L}(\theta) | > 0$. For all $\tilde{\theta} \in \Par$ such that $\| \phi(\tilde{\theta}) - \phi(\theta)\|_{\infty} < \epsilon$ we have
\[\phi_{p, v_L}(\tilde{\theta}) \neq 0.\]
Since $\phi_{p, v_L}(\tilde{\theta}) = \tilde{\theta}_{p_i} \tilde{w}_{v \rightarrow v'} \tilde{\theta}_{p_o}$, this implies in particular that
\[\tilde{w}_{v \rightarrow v'} \neq 0.\]

The proof is similar in the case $v \in B$ and $b_v \neq 0$.
\end{proof}

The following proposition, which states that for any $\theta \in \ParS$, $V_\theta = \psi^\theta(U_\theta)$ is open with respect to the topology induced on $\Sigma_1^*$ by the standard topology of $\RR^{\mathcal{P}\times V_L}$, is necessary to show that $(V_\theta, (\psi^\theta)^{-1})_{\theta \in \ParS}$ is an atlas of $\Sigma_1^*$.

\begin{prop}\label{psi^t(U_t) is a neighborhood of t}
For any $\theta \in \ParS$, for any $\tau \in U_{\theta}$, there exists $\epsilon > 0$ such that
\[\Sigma_1^* \cap B_{\infty}(\psi^{\theta}(\tau),\epsilon) \subset V_\theta.\]
\end{prop}

\begin{proof}
Let us first construct $\epsilon$ and then consider an element of the set on the left of the inclusion and prove it belongs to $V_\theta$. Let $\theta \in \ParS$ and $\tau \in U_\theta$. For all $l \in \lb 1, L-1 \rb$, for all $v \in V_l$, by definition of $F_\theta$ and $\rho_\theta$, we have $\rho_\theta(\tau)_{v \rightarrow s^\theta_{\max}(v)} = w_{v \rightarrow s^\theta_{\max}(v)}$, and since $\theta \not\in S$, by definition of $s^\theta_{\max}$, $ w_{v \rightarrow s^\theta_{\max}(v)} \neq 0$, so according to Lemma \ref{tilde(theta)_v->v' neq 0 close to theta} there exists $\epsilon_v > 0$ such that for all $\tilde{\theta} \in \Par$, 
\[\| \phi(\rho_\theta(\tau)) - \phi(\tilde{\theta})\|_\infty < \epsilon_v \quad \Longrightarrow \quad \tilde{w}_{v \rightarrow s^\theta_{\max}(v)} \neq 0.\]
Let $\epsilon = \min_{v \in V_1 \cup \dots \cup V_{L-1}} \epsilon_v$.

Let us now show the inclusion: let $\tilde{\theta} \in \ParS$ such that $\| \phi(\rho_\theta(\tau)) - \phi(\tilde{\theta})\|_\infty < \epsilon$, and let us show that $\phi(\tilde{\theta}) \in V_\theta$. Notice first that for all $l \in \lb 1, L-1 \rb$ and $v \in V_l$, by definition of $\epsilon$, $w_{v \rightarrow s^\theta_{\max}(v)} \neq 0$ and \mbox{$\tilde{w}_{v \rightarrow s^\theta_{\max}(v)} \neq 0$}. We are going to define $\tilde{\tau} \in U_\theta$ such that $\rho_\theta(\tilde{\tau}) \overset{R}{\sim} \tilde{\theta}$, so that using Proposition \ref{equivalence implies equal phi}, $\psi^\theta(\tilde{\tau}) = \phi(\tilde{\theta})$.

Let us define recursively a family $(\lambda^0, \dots, \lambda^L) \in (\RR^*)^{V_0} \times \dots \times (\RR^*)^{V_L}$ as follows:
\begin{itemize}
\item we define $\lambda^L = \One_{V_L}$;
\item for all $l \in \lb 1, L-1 \rb$, for all $v \in V_l$, we define
\begin{equation}\label{proof V_theta is open}
    \lambda^l_v = \frac{\tilde{w}_{v \rightarrow s^\theta_{\max}(v)}}{w_{v \rightarrow s^\theta_{\max}(v)}} \lambda^{l+1}_{s^\theta_{\max}(v)} .
\end{equation}
\item we define finally $\lambda^0 = \One_{V_0}$.
\end{itemize}
Note that for all $l \in \lb 0, L \rb$ and for all $v \in V_l$, $\lambda^l_v \neq 0$. Also note that for all $l \in \lb 2, L \rb$, for all $v \in V_{l-1}$, reformulating \eqref{proof V_theta is open} in a way that will be useful later, we have
\begin{equation}\label{proof psi^t(U_t) is a neighborhood of t-1}
\frac{\lambda^{l}_{s^\theta_{\max}(v)}}{\lambda^{l-1}_v} = \frac{w_{v \rightarrow s^\theta_{\max}(v)}}{\tilde{w}_{v \rightarrow s^\theta_{\max}(v)}}.
\end{equation}

We then define $\tilde{\tau} \in \ParF$ by:
\begin{itemize}
\item for all $l \in \lb 1, L \rb$, for all $(v,v') \in ( V_{l-1} \times V_l )\cap F_\theta$, 
\begin{equation}\label{proof psi^t(U_t) is a neighborhood of t-2}
\tilde{\tau}_{v \rightarrow v'} = \frac{\lambda^l_{v'}}{\lambda^{l-1}_v}\tilde{w}_{v \rightarrow v'};
\end{equation}
\item for all $l \in \lb 1, L \rb$, for all $v \in V_l$, 
\begin{equation}\label{proof psi^t(U_t) is a neighborhood of t-3}
\tilde{\tau}_v = \lambda^l_v \tilde{b}_v.
\end{equation}
\end{itemize}

Let us show $\rho_\theta(\tilde{\tau}) \rsim \tilde{\theta}$. Indeed, let $l \in \lb 1, L \rb$ and let $(v,v') \in V_{l-1} \times V_l$.
If $v \in V_0$ or $v \in V_1 \cup \dots \cup V_{L-1}$ and $v' \neq s^\theta_{\max}(v)$, then by definition \eqref{definition of F_theta-main} of $F_\theta$, we have $v \rightarrow v' \in F_\theta$, so using \eqref{definition of rho_theta-main} and \eqref{proof psi^t(U_t) is a neighborhood of t-2} we have
\begin{equation}\label{proof psi^t(U_t) is a neighborhood of t-4}
\rho_\theta(\tilde{\tau})_{v \rightarrow v'} = \tilde{\tau}_{v \rightarrow v'} = \frac{\lambda^l_{v'}}{\lambda^{l-1}_v} \tilde{w}_{v \rightarrow v'}.
\end{equation}
If $v \in V_1 \cup \dots \cup V_{L-1}$ and $v' = s^\theta_{\max}(v)$, then by definition \eqref{definition of F_theta-main} of $F_\theta$, we have $v \rightarrow v' \in E \backslash F_\theta$, and since in that case, $l \geq 2$, using \eqref{definition of rho_theta-main} and \eqref{proof psi^t(U_t) is a neighborhood of t-1}, we see that
\begin{equation}\label{proof psi^t(U_t) is a neighborhood of t-5}
\rho_\theta(\tilde{\tau})_{v \rightarrow v'} = w_{v \rightarrow v'} = \frac{\lambda^l_{v'}}{\lambda^{l-1}_v} \tilde{w}_{v \rightarrow v'}.
\end{equation}
If $v \in B$, using \eqref{definition of rho_theta-main} and \eqref{proof psi^t(U_t) is a neighborhood of t-3}, we have
\begin{equation}\label{proof psi^t(U_t) is a neighborhood of t-6}
   \rho_\theta(\tilde{\tau})_v = \tilde{\tau}_v = \lambda^l_v \tilde{b}_v.
\end{equation}

Equations \eqref{proof psi^t(U_t) is a neighborhood of t-4}, \eqref{proof psi^t(U_t) is a neighborhood of t-5} and \eqref{proof psi^t(U_t) is a neighborhood of t-6} prove that 
\[\rho_\theta(\tilde{\tau}) \rsim \tilde{\theta}.\]

Using Corollary \ref{ParS is stable by r equivalence}, since $\tilde{\theta} \in \ParS$ and $\rho_\theta(\tilde{\tau}) \rsim \tilde{\theta}$, we also have $\rho_\theta(\tilde{\tau}) \in \ParS$. Since, by definition, $U_\theta = \rho_\theta^{-1}\left( \ParS \right)$, we have $\tilde{\tau} \in U_\theta$. We have shown
\[\Sigma_1^* \cap B_\infty(\psi^\theta(\tau), \epsilon) \ \subset \ V_\theta.\]
\end{proof}

The following proposition is necessary in order to show that $(V_\theta, (\psi^\theta)^{-1})_{\theta \in \ParS}$ is an atlas of $\Sigma_1^*$.
\begin{prop}\label{the differential of psi^theta is injective}
For all $\theta \in \ParS$, the function $\psi^{\theta}$ is $C^{\infty}$ and its differential $D{\psi^{\theta}}(\tau)$  is injective for all $\tau \in U_{\theta}$.
\end{prop}

\begin{proof}
Let $\theta \in \ParS$. First of all, $\psi^\theta$ is a polynomial function as a composition of $\phi$ and $\rho_\theta$ which are both polynomial functions. So, $\psi^\theta$ is $C^\infty$.

In order to show the injectivity of the differential $D{\psi^{\theta}}(\tau)$ for all $\tau \in U_\theta$, let us compute the partial derivatives of $\psi^\theta_{p, v_L}(\tau)$. Let $p \in \mathcal{P}$ and $v_L \in V_L$. Using the definition of $\psi^\theta$ and $\phi$, three cases are possible.
\begin{itemize}
\item[Case 1.] The path $p$ is of the form $(v_0, v_1, \dots, v_{L-1})$. We have
\[\psi^\theta_{p, v_L}(\tau) =  \rho_\theta(\tau)_{v_0 \rightarrow v_1} \dots \rho_\theta(\tau)_{v_{L-1} \rightarrow v_L}.\]
\item[Case 2.] The path $p$ is of the form $(v_l, \dots, v_{L-1})$ with $l \in \lb 1, L-1 \rb$. We have, for all $\tau \in U_\theta$, 
\[\psi^\theta_{p,v_L}(\tau) = \tau_{v_l} \rho_\theta(\tau)_{v_l \rightarrow v_{l+1}} \dots \rho_\theta(\tau)_{v_{L-1} \rightarrow v_L}.\]
\item[Case 3.] For $p = \beta$, we have, for all $\tau \in U_\theta$, 
\[\psi^\theta_{p, v_L}(\tau) = \tau_{v_L}.\]
\end{itemize}

Let $(v,v') \in F_\theta$, and let us compute $\frac{\partial \psi^\theta_{p,v_L}}{\partial \tau_{v \rightarrow v'}}(\tau)$.
\begin{itemize}
\item[Case 1.] We have $p = (v_0, \dots, v_{L-1}) \in \mathcal{P}_0$. If $\{v,v'\} \subset \{v_0, \dots, v_L \}$, there exists $l \in \lb 0 , L-1 \rb$ such that $(v,v') = (v_l, v_{l+1})$, in which case, since $(v,v') \in F_\theta$, $\rho_\theta(\tau)_{v_{l} \rightarrow v_{l+1}} = \tau_{v_{l} \rightarrow v_{l+1}}$ and
\begin{equation}\label{proof Dpsi is injective-1}
    \frac{\partial \psi^\theta_{p, v_L}}{\partial \tau_{v \rightarrow v'}}(\tau) = \prod_{\substack{k \in \lb 0 , L-1 \rb \\ k \neq l }} \rho_\theta(\tau)_{v_k \rightarrow v_{k+1}}.
\end{equation}
Otherwise if $\{v,v'\} \not\subset \{v_0, \dots, v_L\}$, 
\[\frac{\partial \psi^\theta_{p,v_L}}{\partial \tau_{v \rightarrow v'}}(\tau) =0.\]
\item[Case 2.] We have $p = (v_l, \dots, v_{L-1}) \in \mathcal{P}_l$, for $l \in \lb 1 , L-1 \rb$. If $\{v,v'\} \subset \{v_l, \dots, v_L\}$, there exists $l' \in \lb l, L-1 \rb$ such that $(v,v') = (v_{l'}, v_{l'+1})$, in which case, since $(v,v') \in F_\theta$, $\rho_\theta(\tau)_{v_{l'} \rightarrow v_{l'+1}} = \tau_{v_{l'} \rightarrow v_{l'+1}}$ and 
\begin{equation}\label{proof Dpsi is injective-2}
    \frac{\partial \psi^\theta_{p, v_L}}{\partial \tau_{v \rightarrow v'}}(\tau) = \tau_{v_l} \prod_{\substack{k \in \lb l, L-1 \rb \\ k \neq l'}} \rho_\theta(\tau)_{v_k \rightarrow v_{k+1}}.
\end{equation}
Otherwise if $\{v,v'\} \not\subset \{v_l, \dots, v_L\}$,
\[\frac{\partial \psi^\theta_{p,v_L}}{\partial \tau_{v \rightarrow v'}}(\tau) = 0.\]
\item[Case 3.] We have $p = \beta$. In that case, we have
\[\frac{\partial \psi^\theta_{p,v_L}}{\partial \tau_{v \rightarrow v'}}(\tau) = 0.\]
\end{itemize}

Now let $v \in B$, and let us compute $\frac{\partial \psi^\theta_{p,v_L}}{\partial \tau_v}(\tau)$.
\begin{itemize}
\item[Case 1.] We have $p = (v_0, \dots, v_{L-1}) \in \mathcal{P}_0$ and
\[\frac{\partial \psi^\theta_{p,v_L}}{\partial \tau_v}(\tau) =0.\]
\item[Case 2.] We have $p = (v_l, \dots, v_{L-1}) \in \mathcal{P}_l$ for $l \in \lb 1, L-1 \rb$. If $v=v_l$, then
\[\frac{\partial \psi^\theta_{p,v_L}}{\partial \tau_v}(\tau) = \prod_{k \in \lb l, L-1 \rb} \rho_\theta(\tau)_{v_k \rightarrow v_{k+1}}.\]
If $v \neq v_l$, 
\[\frac{\partial \psi^\theta_{p,v_L}}{\partial \tau_v}(\tau) = 0.\]
\item[Case 3.] We have $p = \beta$ and
\[\frac{\partial \psi^\theta_{p,v_L}}{\partial \tau_v}(\tau) = \begin{cases} 1 & \text{if } v = v_L \\ 0 & \text{if } v \neq v_L. \end{cases}\]
\end{itemize}

Now that we know the partial derivatives, let us show $D\psi^\theta(\tau)$ is injective for all $\tau \in U_\theta$. Let $\tau \in U_\theta$ and let $h \in \ParF$ such that
\[D\psi^\theta(\tau) \cdot h = 0.\]
We need to prove that $h=0$.

Let us show first that for all $v \in B$, $h_v = 0$. Let $l \in \lb 1, L - 1 \rb$, and let $v_l \in V_l$ so that $v_l$ is arbitrary in $B \backslash V_L$. Let us define $v_{l+1} = s^\theta_{max}(v_l)$, then $v_{l+2} = s^\theta_{\max}(v_{l+1})$, and so on up to $v_L = s^\theta_{max}(v_{L-1})$. Let us denote $p = (v_l, \dots, v_{L-1})$. We have
\[\psi^\theta_{p,v_L} (\tau) = \tau_{v_l} w_{v_l \rightarrow v_{l+1}} \dots w_{v_{L-1} \rightarrow v_L},\]
so
\[\left[D\psi^\theta(\tau) \cdot h \right]_{p,v_L} = \frac{\partial \psi^\theta_{p, v_L}}{\partial \tau_{v_l}}(\tau) h_{v_l} = w_{{v_l} \rightarrow v_{l+1}}  \dots w_{v_{L-1} \rightarrow v_{L}} h_{v_l}.\]
Since $\left[D\psi^\theta(\tau) \cdot h \right]_{p,v_L} = 0$ and $w_{v_l \rightarrow v_{l+1}} \dots w_{v_{L-1} \rightarrow v_{L}} \neq 0$, we conclude that $h_{v_l} = 0$. Now let $v_L \in V_L$. We consider $p = \beta$ and we have
\[\left[D\psi^\theta(\tau) \cdot h \right]_{p,v_L} = h_{v_L}.\]
Since $\left[D\psi^\theta(\tau) \cdot h \right]_{p,v_L} = 0$, we also conclude in that case that $h_{v_L} = 0$.

Let us now show that for all $(v,v')\in F_\theta$, $h_{v \rightarrow v'} = 0$. Let $l \in \lb 1, L \rb$ and let $(v_{l-1},v_l) \in (V_{l-1} \times V_l) \cap F_\theta$ so that $(v_{l-1}, v_l)$ is arbitrary in $F_\theta$. If $l=1$, we define $p_i = (v_{l-1})$ and we have by convention $\theta_{p_i} =1 \neq 0$. If $l> 1$, using Proposition \ref{nonzero input and output paths} there exist $l' \in \lb 0, l-1 \rb$ and a path $p_i = (v_{l'}, \dots, v_{l-1})$ such that $\rho_\theta(\tau)_{p_i} \neq 0$. If $l < L$, we define $v_{l+1}= s^\theta_{\max}(v_{l})$, then $v_{l+2} = s^\theta_{\max}(v_{l+1})$, and so on up to $v_L = s^\theta_{\max}(v_{L-1})$, and we denote $p = p_i \cup (v_{l-1},v_l,  \dots, v_{L-1})$. If $l=L$, we denote $p = p_i$. Let us show the following expression.
\begin{equation}\label{proof Dpsi is injective}
\left[D{\psi^\theta}(\tau) \cdot h \right]_{p,v_L}  =  \sum_{\substack{k \in \lb l', l-1 \rb \\ (v_k, v_{k+1})\in F_\theta}} \frac{\partial \psi^\theta_{p, v_L}}{\partial \tau_{v_k \rightarrow v_{k+1}}}(\tau) h_{v_k \rightarrow v_{k+1}}
\end{equation} 
Indeed, if $l' \geq 1$, we have
\[\psi^\theta_{p,v_L}(\tau) = \tau_{v_{l'}}\prod_{k=l'}^{l-1} \rho_\theta(\tau)_{v_k \rightarrow v_{k+1}} \prod_{k=l}^{L-1} w_{v_k \rightarrow v_{k+1}},\]
with the classical convention that if $l = L$, the product on the right is empty thus equal to $1$.
We thus have
\begin{align*}
\left[D{\psi^\theta}(\tau) \cdot h \right]_{p,v_L} & = \frac{\partial \psi^\theta_{p,v_L}}{\partial \tau_{v_{l'}}}(\tau) h_{v_{l'}} + \sum_{\substack{k \in \lb l', l-1 \rb \\ (v_k, v_{k+1})\in F_\theta}} \frac{\partial \psi^\theta_{p, v_L}}{\partial \tau_{v_k \rightarrow v_{k+1}}}(\tau) h_{v_k \rightarrow v_{k+1}}  \\
& = \sum_{\substack{k \in \lb l', l-1 \rb \\ (v_k, v_{k+1})\in F_\theta}} \frac{\partial \psi^\theta_{p, v_L}}{\partial \tau_{v_k \rightarrow v_{k+1}}}(\tau)  h_{v_k \rightarrow v_{k+1}} ,
\end{align*}
since we have already shown that $h_{v_{l'}} = 0$.

If $l' =0$, we have
\[\psi^\theta_{p,v_L}(\tau) = \prod_{k=0}^{l-1} \rho_\theta(\tau)_{v_k \rightarrow v_{k+1}} \prod_{k=l}^{L-1} w_{v_k \rightarrow v_{k+1}},\]
with the same convention that when $l=L$ the product on the right is equal to $1$, so again
\begin{align*}
\left[D{\psi^\theta}(\tau) \cdot h \right]_{p,v_L} & =  \sum_{\substack{k \in \lb 0, l-1 \rb \\ (v_k, v_{k+1})\in F_\theta}} \frac{\partial \psi^\theta_{p, v_L}}{\partial \tau_{v_k \rightarrow v_{k+1}}}(\tau)  h_{v_k \rightarrow v_{k+1}}.
\end{align*}
This concludes the proof of \eqref{proof Dpsi is injective}.

We can now show by induction the following statement, for $l \in \lb 0 , L \rb$.
\begin{equation}\tag{$H_l$}
\forall l' \in \lb 1, l \rb , \ \forall (v,v') \in \left( V_{l'-1}\times V_{l'}\right) \cap F_\theta, \ h_{v \rightarrow v'} = 0.
\end{equation}
Since $\lb 1,0 \rb = \emptyset$, $H_0$ is trivially true. Now let $l \in \lb 1, L \rb$ and suppose $H_{l-1}$ is true. We consider $(v_{l-1},v_l) \in (V_{l-1} \times V_l) \cap F_\theta$, and $l' \in \lb 0, l\rb$, $p_i$ and $p$ just as before. Since for all $k \in \lb 0, l-2 \rb$, the induction hypothesis guarantees that $h_{v_k \rightarrow v_{k+1}}=0$,  \eqref{proof Dpsi is injective} becomes
\begin{align*}\left[D{\psi^\theta}(\tau) \cdot h \right]_{p,v_L} = \frac{\partial \psi^\theta_{p, v_L}}{\partial \tau_{v_{l-1} \rightarrow v_{l}}}(\tau) h_{v_{l-1} \rightarrow v_{l}}.
\end{align*}
Using \eqref{proof Dpsi is injective-1} and \eqref{proof Dpsi is injective-2}, we obtain
\begin{align*}
 \left[D{\psi^\theta}(\tau) \cdot h \right]_{p,v_L}   = \begin{cases} \rho_\theta(\tau)_{p_i} w_{v_l \rightarrow v_{l+1}} \dots w_{v_{L-1} \rightarrow v_L} h_{v_{l-1} \rightarrow v_l} & \text{if } l < L \\ \rho_\theta(\tau)_{p_i} h_{v_{l-1} \rightarrow v_l} & \text{if } l=L.\end{cases}
\end{align*}
Since $\rho_\theta(\tau)_{p_i} \neq 0$, and for $l < L$, $ w_{v_l \rightarrow v_{l+1}} \dots w_{v_{L-1} \rightarrow v_L} \neq 0$, we conclude that $h_{v_{l-1} \rightarrow v_l} = 0$ and that $H_l$ holds.

This induction leads to the conclusion that $h = 0$ and $D\psi^\theta(\tau)$ is injective.
\end{proof}

We are now equipped to prove Theorem \ref{Sigma_1^* is a smooth manifold-main}, which we restate here.

\begin{thm}\label{Sigma_1^* is a smooth manifold}
$\Sigma^1_*$ is a smooth manifold of $\RR^{\mathcal{P} \times V_L}$ of dimension 
\[|F_{\theta}| + |B| = N_0 N_1 + N_1 N_2 + \dots + N_{L-1}N_L + N_L,\]
and the family $(V_\theta, (\psi^\theta)^{-1})_{\theta \in \ParS}$ is an atlas.
\end{thm}
\begin{proof}

Our goal is to show that the family $(V_\theta, (\psi^\theta)^{-1})_{\theta \in \ParS}$ is a smooth atlas, which will show that $\Sigma_1^*$ is a smooth manifold.

We already know from Proposition \ref{psi^t(U_t) is a neighborhood of t} that for any $\theta \in \ParS$, $V_\theta$ is an open subset of $\Sigma_1^*$ and from Proposition \ref{psi^theta is a homeomorphism} that $(\psi^\theta)^{-1}$ is a homeomorphism from $V_\theta$ onto $U_\theta$. Since for any $\theta \in \ParS$, $\thF \in U_\theta$, we have $\phi(\theta) = \psi^\theta(\thF) \in V_\theta$ which shows that $(V_\theta)_{\theta \in \ParS}$ covers $\Sigma_1^*$.

Let $\theta, \tilde{\theta} \in \ParS$, let us show that the transition map
\[ (\psi^\theta)^{-1} \circ \psi^{\tilde{\theta}} \ : \ {(\psi^{\tilde{\theta}})^{-1}(V_\theta \cap V_{\tilde{\theta}})} \ \rightarrow \ {(\psi^{\theta})^{-1}(V_\theta \cap V_{\tilde{\theta}})}  \]
is smooth.

Let $\tau_0 \in U_{\tilde{\theta}}$ such that $\tau_0 \in {(\psi^{\tilde{\theta}})^{-1}(V_\theta \cap V_{\tilde{\theta}})}$. We are going to show that the function $(\psi^\theta)^{-1} \circ \psi^{\tilde{\theta}}$ is $C^\infty$ in a neighborhood of $\tau_0$.

For ease of reading, let us denote $\psi^{\tilde{\theta}}(\tau_0)$ by $\eta_0$. By definition, $\eta_0 \in V_\theta \cap V_{\tilde{\theta}}$. In particular, since $\eta_0 \in V_\theta$, we can define $\tau_1 = (\psi^\theta)^{-1}(\eta_0)$. See Figure \ref{chart figure} for a representation.

\begin{figure}
    \centering
    \includegraphics[scale=0.6]{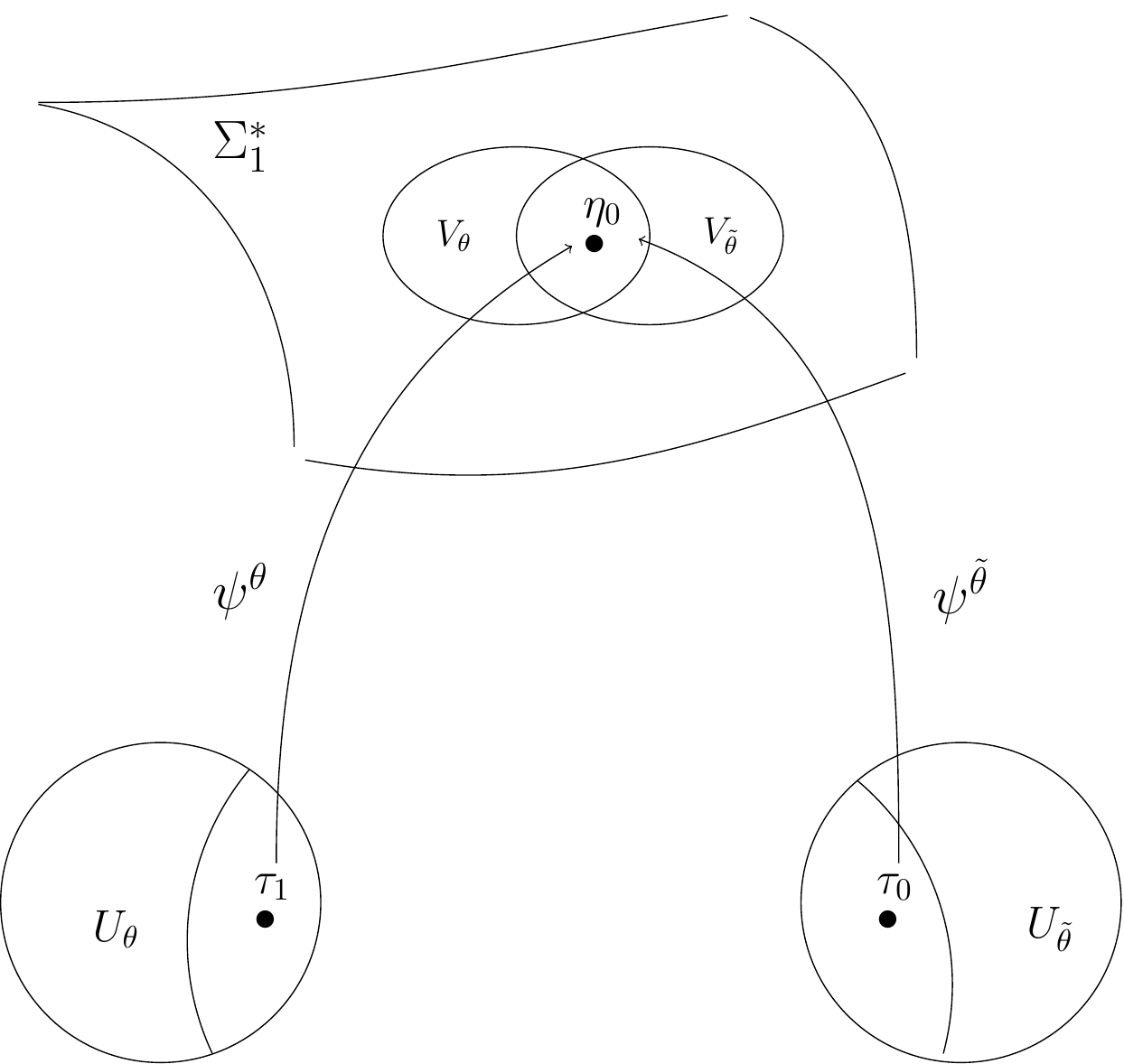}
    \caption{The points $\eta_0$, $\tau_0$, $\tau_1$ and the inverse charts $\psi^\theta$ and $\psi^{\tilde{\theta}}$.}
    \label{chart figure}
\end{figure}

Let $T = \Im D\psi^\theta(\tau_1)$, and let us consider a linear subspace $G$ such that $T \oplus G = \RR^{\mathcal{P}\times V_L}$. Let $N_C = |\mathcal{P}| N_L - |F_\theta| - |B| = \dim(G)$. Let $i : \RR^{N_C} \rightarrow G$ be linear and invertible. Let us consider the function
\[\fonction{\varphi_\theta}{U_{\theta}\times \RR^{N_C}}{\RR^{\mathcal{P}\times V_L}}{(\tau,x)}{\psi^\theta(\tau) + i(x).} \]
We are going to show that there exist an open neighborhood $\tilde{U}$ of $(\tau_1,0)$ in $(\RR^{F_\theta} \times \RR^B) \times \RR^{N_C}$ and an open neighborhood $\tilde{V}$ of $\eta_0$ in $\RR^{\mathcal{P} \times V_L}$ such that $\varphi_\theta$ is a $C^\infty$ diffeomorphism from $\tilde{U}$ onto $\tilde{V}$ satisfying
\[ \varphi_\theta\bigg(\left[(\ParF ) \times \{0 \}^{N_C}\right] \cap \tilde{U} \bigg) = \Sigma_1^* \cap \tilde{V}.\]
Let us first show that $\varphi_\theta$ is a $C^\infty$-diffeomorphism from a neighborhood of $(\tau_1,0)$ in $(\RR^{F_\theta} \times \RR^B) \times \RR^{N_C}$ onto a neighborhood of $\eta_0$ in $\RR^{\mathcal{P} \times V_L}$.
As shown in Proposition \ref{the differential of psi^theta is injective}, $\psi^\theta$ is $C^\infty$ and $i$ is a linear function, so $\varphi_\theta$ is $C^\infty$. Let us prove that the differential $D \varphi_\theta (\tau_1,0)$ is injective. For all $(\tau,x) \in \left( \RR^{F_\theta} \times \RR^B \right) \times \RR^{N_C}$,
\[D\varphi_\theta(\tau_1,0) \cdot (\tau, x) = D\psi^\theta(\tau_1) \cdot \tau + i(x).\]
Since $D\psi^\theta(\tau_1) \cdot \tau \in T$, $i(x) \in G$, and $T$ and $G$ are in direct sum, if $D\varphi_\theta(\tau_1,0) \cdot (\tau, g) = 0$, then we have
\[\begin{cases}D\psi^\theta(\tau_1) \cdot \tau = 0 \\ i(x) = 0. \end{cases}\]
Since as shown in Proposition \ref{the differential of psi^theta is injective} $D\psi^\theta(\tau_1)$ is injective, and since $i$ is invertible, we have
\[(\tau,x) = (0,0).\]
Hence, $D\varphi_\theta(\tau_1,0)$ is injective. Since $\dim(\left( \RR^{F_\theta} \times \RR^B \right) \times \RR^{N_C}) = |F_\theta| + |B| + N_C = | \mathcal{P} | N_L$, the differential $D\varphi_\theta(\tau_1,0)$ is bijective. Using the inverse function theorem, there exists an open set $U \subset U_\theta  \times \RR^{N_C}$ containing $(\tau_1,0)$, an open set $V \subset \RR^{\mathcal{P} \times V_L}$ containing $\eta_0$ such that $\varphi_\theta$ is a $C^{\infty}$-diffeomorphism from $U$ onto $V$.

We have 
\[\varphi_\theta\bigg(\left[(\ParF ) \times \{0 \}^{N_C}\right] \cap U \bigg) \quad \subset \quad V_\theta \cap V .\]
In fact, if $V$ is small enough, this inclusion is an equality. We are going to construct open subsets $\tilde{U} \subset U$ and $\tilde{V} \subset V$ so that it is the case.
Let us define 
\[O = \{ \tau \in U_\theta, (\tau, 0) \in U \}.\]
Since $U$ is an open set containing $(\tau_1,0)$, $O$ is an open set containing $\tau_1 = (\psi^\theta)^{-1}(\eta_0)$. Since, according to Proposition \ref{psi^theta is a homeomorphism}, $\psi^\theta$ is a homeomorphism, $\psi^\theta(O)$ is an open subset of $V_\theta$ so there exists $\epsilon > 0$ such that 
\begin{equation}\label{proof Sigma_1^* is a manifold - first inclusion}
V_\theta \cap B_\infty(\eta_0, \epsilon) \ \subset \ \psi^\theta( O).
\end{equation}
We can now define $\tilde{V} = V \cap B_\infty ( \eta_0, \epsilon)$, and $\tilde{U} = \{ ( \tau, x) \in U, \ \varphi_\theta(\tau,x) \in \tilde{V} \}$, which are open sets such that $(\tau_1,0) \in \tilde{U}$, $\eta_0 \in \tilde{V}$, and $\varphi_\theta$ is a $C^\infty$-diffeomorphism from $\tilde{U}$ onto $\tilde{V}$. Let us show that
\begin{equation}\label{proof Sigma_1^* is a manifold - equality}
\varphi_\theta\bigg(\left[(\ParF ) \times \{0 \}^{N_C}\right] \cap \tilde{U} \bigg) = V_\theta \cap \tilde{V}.
\end{equation}
The direct inclusion is immediate: if $(\tau, 0) \in \left[(\ParF ) \times \{0 \}^{N_C}\right] \cap \tilde{U}$, then
\[\varphi_\theta(\tau, 0) = \psi^\theta(\tau) \in V_\theta \cap \tilde{V}.\]
For the reciprocal inclusion, if $\tau \in U_\theta$ is such that $\psi^\theta(\tau) \in V_\theta \cap \tilde{V}$, then by definition of $\epsilon$ and $\tilde{V}$, \eqref{proof Sigma_1^* is a manifold - first inclusion} guarantees, since $\psi^\theta$ is injective, that $\tau \in O$. By definition of $O$, we have $(\tau,0) \in U$, and since
\[\varphi_\theta(\tau, 0) = \psi^\theta (\tau) \in \tilde{V},\]
this shows $(\tau, 0) \in \tilde{U}$. This shows the reciprocal inclusion, and thus \eqref{proof Sigma_1^* is a manifold - equality} holds.

Let us now define
\[\fonction{P_\theta}{\ParF \times \RR^{N_C}}{\ParF}{(\tau, x)}{\tau}\]
the restriction to the first component, and let us observe that over $V_\theta \cap \tilde{V}$, we have 
\begin{equation}\label{proof Sigma_1^* is a manifold - inversion of varphi_theta}
    P_\theta \circ (\varphi_\theta)^{-1} = (\psi^\theta)^{-1}.
\end{equation}
Indeed, if $\eta \in V_\theta \cap \tilde{V}$, then by \eqref{proof Sigma_1^* is a manifold - equality} there exists $\tau \in U_\theta$ such that $(\tau, 0) \in \tilde{U}$ and $\varphi_\theta(\tau, 0) = \eta$. Since  $\varphi_\theta(\tau, 0) = \psi^\theta(\tau)$, this shows that $\tau = (\psi^\theta)^{-1}(\eta)$ and thus 
\[(\psi^\theta)^{-1}(\eta) = P_\theta(\tau,0) =  P_\theta \circ (\varphi_\theta)^{-1}(\eta) .\]

Now recall that $\eta_0 = \psi^{\tilde{\theta}}(\tau_0)$. By continuity of $\psi^{\tilde{\theta}}$, there exists $\epsilon ' >0$ such that $B_\infty ( \tau_0, \epsilon') \subset (\psi^{\tilde{\theta}})^{-1}(V_\theta \cap V_{\tilde{\theta}})$ and 
\[ \psi^{\tilde{\theta}} (B_\infty(\tau_0, \epsilon ')) \subset \tilde{V}.\]
For any $\tau \in B_\infty(\tau_0, \epsilon ')$, we have $\psi^{\tilde{\theta}}(\tau) \in V_\theta \cap \tilde{V}$ so, as we just proved with \eqref{proof Sigma_1^* is a manifold - inversion of varphi_theta}, $(\psi^\theta)^{-1} \circ \psi^{\tilde{\theta}}(\tau) = P_\theta \circ (\varphi_\theta)^{-1} \circ \psi^{\tilde{\theta}}(\tau)$. Since the functions $\psi^{\tilde{\theta}}$, $(\varphi_\theta)^{-1}$ and $P_\theta$ are all $C^\infty$, we conclude that the transition map $(\psi^\theta)^{-1} \circ \psi^{\tilde{\theta}}$ is $C^\infty$ over $B_\infty(\tau_0, \epsilon')$, for all $\tau_0 \in (\psi^{\tilde{\theta}})^{-1}(V_\theta \cap V_{\tilde{\theta}})$. We conclude that $(\psi^\theta)^{-1} \circ \psi^{\tilde{\theta}}$ is $C^\infty$ over $(\psi^{\tilde{\theta}})^{-1}(V_\theta \cap V_{\tilde{\theta}})$.

We have showed that $(V_\theta, (\psi^{\theta})^{-1})_{\theta \in \ParS}$ is a smooth atlas, and thus that $\Sigma_1^*$ is a smooth submanifold of $\RR^{\mathcal{P} \times V_L}$. As computed in \eqref{dimension of Sigma_1^*}, its dimension is
\[|F_\theta| + |B| = N_0N_1 + N_1 N_2  + \dots + N_{L-1}N_L + N_L.\]
\end{proof}

\section{Conditions of local identifiability}\label{Conditions of local identifiability-sec}

Let us restate (using Definition \ref{def:local:identif}) and prove Theorem \ref{local identifiability using Sigma1*-main}.
\begin{thm}\label{local identifiability using Sigma1*}
For any $X \in \RR^{n \times V_0}$ and $\theta \in \ParSX$, the two following statements are equivalent.
\begin{itemize}
\item[i)]$\theta$ is locally identifiable from $X$.

\item[ii)]There exists $\epsilon > 0$ such that $B_\infty(\phi(\theta), \epsilon) \cap \Sigma_1^* \cap N(X,\theta) = \{\phi(\theta)\}$.
\end{itemize}
\end{thm}

\begin{proof}
\item[$i) \Rightarrow ii)$] Suppose $i)$ is satisfied for some $\epsilon_1 > 0$. We first construct $\epsilon' > 0$ and then consider $\eta \in B_\infty(\phi(\theta), \epsilon') \cap \Sigma_1^* \cap N(X,\theta)$, and we prove that $\eta = \phi(\theta)$. Since $\theta \in \ParSX$ and since, according to Proposition \ref{alpha_X is piecewise-constant}, $\Delta_X$ is closed, there exists $\epsilon_2 > 0$ such that for any $\tilde{\theta} \in B_\infty(\theta, \epsilon_2)$, 
\[\alpha(X, \theta) = \alpha(X , \tilde{\theta}),\]
i.e.
\begin{equation*}
A(X,\theta) = A(X ,\tilde{\theta}).
\end{equation*}
Consider $\epsilon = \min(\epsilon_1, \epsilon_2)$. Since, according to Proposition \ref{psi^theta is a homeomorphism}, $\rho_\theta \circ (\psi^\theta)^{-1}$ is continuous at $\phi(\theta) \in \psi^\theta(U_\theta)$, and since $\rho_\theta \circ ( \psi^\theta)^{-1} ( \phi ( \theta)) = \rho_\theta  (\thF) = \theta$, there exists $\epsilon' > 0$ such that for all $\tau \in U_\theta$, 
\begin{equation}\label{proof CNS local identifiability-1}
    \| \psi^\theta(\tau) - \phi(\theta) \|_{\infty} < \epsilon' \quad \Longrightarrow \quad \| \rho_\theta(\tau) - \theta \|_\infty = \| \rho_\theta \circ (\psi^\theta)^{-1}(\psi^\theta(\tau)) - \rho_\theta \circ ( \psi^\theta)^{-1} ( \phi ( \theta)) \|_\infty < \epsilon.
\end{equation}
Since $\phi(\theta) = \psi^\theta(\thF)$, Proposition \ref{psi^t(U_t) is a neighborhood of t} guarantees that, modulo a decrease of $\epsilon '$, we can assume that 
\begin{equation}\label{proof CNS local identifiability-2}
    B_\infty(\phi(\theta), \epsilon') \cap \Sigma_1^* \subset \psi^\theta(U_\theta) .
\end{equation}

Now let $\eta \in B_\infty(\phi(\theta), \epsilon') \cap \Sigma_1^* \cap N(X,\theta)$. Let us prove that $\eta = \phi(\theta)$. Using \eqref{proof CNS local identifiability-2}, there exists $\tau \in U_\theta$ such that $\eta = \psi^\theta (\tau)$. Since $\|\phi(\theta) - \eta \|_\infty < \epsilon'$, we have using \eqref{proof CNS local identifiability-1}
\begin{equation}\label{proof abstract CNS - 1}
\|\rho_\theta(\tau) - \theta \|_\infty < \epsilon.
\end{equation}
Since $\epsilon < \epsilon_2$, we have
\begin{equation}\label{A(X,theta) = A(X, rho(tau))}
A(X, \theta) = A(X, \rho_\theta(\tau)).
\end{equation}
Since $\psi^\theta(\tau) = \eta \in N(X, \theta)$, we have by definition of $N(X, \theta)$ that $\psi^\theta(\tau) - \phi(\theta) \in \Ker A(X, \theta)$, so
\begin{equation}\label{A(X, theta) rho_theta = A(X, theta) phi(theta)}
A(X, \theta) \cdot \psi^\theta(\tau) = A(X, \theta) \cdot \phi(\theta)
\end{equation}
Using successively \eqref{fundamental prop of alpha-matrix version-main}, \eqref{A(X,theta) = A(X, rho(tau))}, \eqref{A(X, theta) rho_theta = A(X, theta) phi(theta)} and \eqref{fundamental prop of alpha-matrix version-main} again, we have
\begin{equation*}
\begin{aligned}
f_{\rho_\theta(\tau)}(X) & = A(X, \rho_\theta(\tau)) \cdot \phi(\rho_\theta(\tau)) \\
& = A(X, \theta) \cdot \phi(\rho_\theta(\tau)) \\
& = A(X, \theta) \cdot \phi(\theta) \\
& = f_\theta(X).
\end{aligned}
\end{equation*}
Since the hypothesis $i)$ holds for $\epsilon_1$, using \eqref{proof abstract CNS - 1} and the fact that $\epsilon < \epsilon_1$, we have 
\[\theta \sim \rho_\theta(\tau).\]
We conclude using Proposition \ref{equivalence implies equal phi} that 
\[\eta = \phi( \rho_\theta(\tau)) = \phi(\theta),\]
which shows 
\[B_\infty(\phi(\theta), \epsilon') \cap \Sigma_1^* \cap N(X,\theta) \ \subset \ \{\phi(\theta)\}.\]
The converse inclusion trivially holds and therefore $ii)$ holds.

\item[$ii) \Rightarrow i)$] 
Suppose $ii)$ is satisfied for some $\epsilon' > 0$. 

We first construct $\epsilon$ and prove $i)$ holds. Since $\theta \in \ParSX$, using Proposition \ref{alpha_X is piecewise-constant}, there exists $\epsilon_1 > 0$ such that for all $\tilde{\theta} \in B_\infty(\theta, \epsilon_1)$, 
\[ \alpha(X, \theta) = \alpha(X, \tilde{\theta}),\]
i.e.
\begin{equation}\label{A(X,theta) = A(X, rho(tau)) 2}
A(X, \theta) = A(X, \tilde{\theta}).
\end{equation}
Since $\phi$ is continuous, there exists $\epsilon_2 > 0$ such that 
\[\|\theta - \tilde{\theta}\|_{\infty} < \epsilon_2 \quad \Longrightarrow \quad \| \phi(\theta) - \phi(\tilde{\theta})\|_\infty < \epsilon'.\]
Using Proposition \ref{rescaling equivalent and close implies equivalent}, there exists $\epsilon_3 > 0$ such that 
\[\theta \overset{R}{\sim} \tilde{\theta} \ \text{and} \ \|\theta - \tilde{\theta}\|_\infty < \epsilon_3 \quad \Longrightarrow \quad \theta \sim \tilde{\theta}.\]
Since $\theta \not\in S$ and $S$ is closed, there exists $\epsilon_4 > 0$ such that for all $\tilde{\theta} \in \Par$, if $\|\theta - \tilde{\theta}\|_\infty < \epsilon_4$, then 
\[\tilde{\theta} \not\in S.\]

Let $\epsilon = \min(\epsilon_1, \epsilon_2, \epsilon_3, \epsilon_4)$. Let $\tilde{\theta} \in B_\infty(\theta, \epsilon)$, and suppose 
\[f_\theta(X) = f_{\tilde{\theta}}(X).\]
Let us prove that $\theta \sim \tilde{\theta}$. Reformulating the above equality using \eqref{fundamental prop of alpha-matrix version-main} for both sides, and using the definition of $A$ given in the beginning of Section \ref{main results-sec-main}, we have
\begin{equation*}
\begin{aligned}
A(X, \theta) \cdot \phi(\theta)  = A(X, \tilde{\theta}) \cdot \phi(\tilde{\theta}).
\end{aligned}
\end{equation*}
Since $\| \theta - \tilde{\theta} \|_\infty < \epsilon \leq \epsilon_1$, we have the equality \eqref{A(X,theta) = A(X, rho(tau)) 2} and thus
\[A(X, \theta) \cdot \phi(\theta)  =  A(X, \theta) \cdot \phi(\tilde{\theta}).\]
In other words, $\phi(\tilde{\theta}) - \phi(\theta) \in \Ker A(X, \theta)$. 
Since $\epsilon < \epsilon_4$, $\phi(\tilde{\theta}) \in \Sigma_1^*$. Since $\epsilon < \epsilon_2$, $\phi(\tilde{\theta}) \in B_\infty(\phi(\theta), \epsilon')$. 
Summarizing, 
\[\phi(\tilde{\theta}) \in B_\infty(\phi(\theta), \epsilon') \cap \Sigma_1^* \cap N(X,\theta) ,\]
and using the hypothesis $ii)$, we conclude that 
\[\phi(\tilde{\theta}) = \phi(\theta).\]
By Proposition \ref{equal phi implies resc equivalence}, we have $\theta \overset{R}{\sim} \tilde{\theta}$, and since $\epsilon < \epsilon_3$, we conclude that
\[\theta \sim \tilde{\theta}.\]

\end{proof}

We are now going to prove Theorems \ref{Necessary condition-main} and \ref{Sufficient condition-main}, which we restate as Theorems \ref{Necessary condition} and \ref{Sufficient condition} respectively (using Definition \ref{def:local:identif}).

\begin{thm}[Necessary condition]\label{Necessary condition}
Let $X \in \RR^{n \times V_0}$ and $\theta \in \ParSX$. If $C_N$ is not satisfied, then $\theta$ is not locally identifiable from $X$ (thus not globally identifiable).
\end{thm}

\begin{thm}[Sufficient condition]\label{Sufficient condition}
Let $X \in \RR^{n \times V_0}$ and $\theta \in \ParSX$. If $C_S$ is satisfied, then $\theta$ is locally identifiable from $X$.
\end{thm}

To prove the theorems, we need to prove first the following lemmas.

\begin{lem}\label{lemma dim(H)}
Let us denote by $T = \Im D\psi^\theta(\thF)$ the direction of the tangent plane to $\Sigma_1^*$ at $\phi(\theta)$. Let us denote by $H$ the intersection $\Ker A(X, \theta) \cap T$. We have
\begin{equation}\label{dim(H)}
\dim(H)  = |F_{\theta}| + |B| - R_\Gamma.
\end{equation}
\end{lem}

\begin{proof}

Let $\eta \in T$. There exists $h \in \ParF$ such that $\eta = D\psi^\theta(\thF) \cdot h$. We have the following equivalence:
\begin{equation*}
\begin{aligned}
\eta \in \Ker A(X, \theta) \quad & \Longleftrightarrow \quad  A(X, \theta) \cdot \eta \ = \ 0 \\
& \Longleftrightarrow \quad   A(X, \theta) \circ D\psi^\theta(\thF) \cdot h  \ = \ 0 \\
& \Longleftrightarrow \quad \Gamma (X, \theta) \cdot h  \ = \ 0 \\
& \Longleftrightarrow \quad h  \in \Ker \Gamma (X, \theta).
\end{aligned}
\end{equation*}

This shows that $D\psi^\theta(\thF)^{-1}(\Ker A(X, \theta) \cap T) = \Ker \Gamma (X, \theta) \subset \ParF$.

Since $D{\psi^{\theta}}(\thF)$ is injective, we thus have
\begin{equation*}
\dim(H)  = \dim(\Ker \Gamma (X, \theta)) = |F_{\theta}| + |B| - R_\Gamma.
\end{equation*}
\end{proof}

\begin{lem}\label{lemma necessary condition}
Let $G$ be a supplementary subspace of $\Ker A(X, \theta)$ such that
\begin{equation}\label{H+G=N}
H \oplus G = \Ker A(X, \theta).
\end{equation}
If $R_\Gamma = R_A$, there exist an open set $\mathcal{O} \subset U_\theta \times G $ containing $(\thF, 0)$ and an open set $\mathcal{V} \subset \RR^{\mathcal{P} \times V_L}$ containing $\phi(\theta)$ such that 
\[\fonction{\xi}{\mathcal{O}}{\mathcal{V}}{(\tau, g)}{\psi^\theta(\tau) + g}\]
is a diffeomorphism from $\mathcal{O}$ onto $\mathcal{V}$.
\end{lem}

\begin{proof}
Let us first show that
\begin{equation}\label{T + G = R^(P x V_l)}
T \oplus G = \RR^{\mathcal{P}\times V_L}.
\end{equation}
Indeed, since $\Ker A(X, \theta) = H \oplus G$ and $T \cap \Ker A(X, \theta) = H$, we have $T \cap G = \{0\}$. We of course have 
\begin{equation}\label{T + G subset R^(P x V_l)}
T \oplus G \subset \RR^{\mathcal{P} \times V_L}.
\end{equation}

Let us show that $\dim(G) = \dim(\RR^{\mathcal{P} \times V_L})-\dim(T)$. First note that we have
\begin{equation}\label{dim(N)}
\dim(\Ker A(X, \theta)) = \dim(\RR^{\mathcal{P} \times V_L}) - \rk(A(X, \theta)) = |\mathcal{P}|N_L - R_A.
\end{equation}
Using \eqref{H+G=N} and \eqref{dim(N)}, we have 
\begin{equation*}
\begin{aligned}
\dim(G) & = \dim(\Ker A(X, \theta)) - \dim(H) \\
& = |\mathcal{P}|N_L - R_A - \dim(H).
\end{aligned}
\end{equation*}

Using \eqref{dim(H)} and the hypothesis $R_\Gamma = R_A$ we thus have

\begin{equation*}
\begin{aligned}
\dim(G) & = |\mathcal{P}|N_L - R_A + R_\Gamma - |F_\theta| - |B| \\
& = |\mathcal{P}|N_L - |F_\theta| - |B| \\
& = |\mathcal{P}|N_L - \dim(T),
\end{aligned}
\end{equation*}
where the last equality comes from the injectivity of $D\psi^\theta (\thF)$, shown in Proposition \ref{the differential of psi^theta is injective}. Together with \eqref{T + G subset R^(P x V_l)}, this proves \eqref{T + G = R^(P x V_l)}.

Let us now consider the function
\[\fonction{\xi}{U_\theta \times G}{\RR^{\mathcal{P}\times V_L}}{(\tau, g)}{\psi^\theta(\tau) + g.}\]
For all $(h,g) \in (\ParF) \times G$, we have

\[D\xi (\thF,0)\cdot(h,g)=D\psi^\theta(\thF)h + g. \]

The differential $D\xi (\thF,0)$ is injective. Indeed, if 
\[D\xi (\thF,0)\cdot(h,g) = 0,\]
then since $D\psi^\theta(\thF)h \in T$ and $g \in G$,  we have
\[\begin{cases}D\psi^\theta(\thF)h = 0 \\ g=0,\end{cases}\]
and since $D\psi^\theta(\thF)$ is injective, $h=0$ and $D \xi (\thF, 0)$ is injective. Since, using \eqref{T + G = R^(P x V_l)},
\[\dim(\ParF) + \dim(G) = |\mathcal{P}|N_L,\]
 $D\xi(\thF, 0)$ is bijective. 

We can thus apply the inverse function theorem: there exists an open set $\mathcal{O} \subset U_\theta \times G $ containing $(\thF, 0)$, an open set $\mathcal{V} \subset \RR^{\mathcal{P} \times V_L}$ containing $\phi(\theta)$ such that $\xi$ is a diffeomorphism from $\mathcal{O}$ into $\mathcal{V}$. 
\end{proof}

We can now prove the theorems.

\begin{proof}[Proof of Theorem \ref{Necessary condition}]
If $C_N$ is not satisfied, then we have $R_\Gamma = R_A < |F_{\theta}| + |B|$. We can thus apply Lemma \ref{lemma necessary condition}: there exist an open set $\mathcal{O} \subset U_\theta \times G $ containing $(\thF, 0)$ and an open set $\mathcal{V} \subset \RR^{\mathcal{P} \times V_L}$ containing $\phi(\theta)$ such that 
\[\fonction{\xi}{\mathcal{O}}{\mathcal{V}}{(\tau, g)}{\psi^\theta(\tau) + g}\]
is a diffeomorphism from $\mathcal{O}$ onto $\mathcal{V}$.

Consider $\epsilon > 0$. We define the open set $\tilde{\mathcal{O}} = \mathcal{O} \cap (\psi^\theta)^{-1}(B(\phi(\theta), \epsilon) \times G$ and its image $\tilde{\mathcal{V}} = \xi(\tilde{\mathcal{O}})$. 

Using the computation of $\dim(H)$ shown in Lemma \ref{lemma dim(H)}, we have
\[\dim(H) = |F_{\theta}| + |B| - R_\Gamma > 0,\]
so there exists a nonzero $w \in H$ such that $\phi(\theta) + w \in \tilde{\mathcal{V}}$. Since $\xi$ induces a diffeomorphism from $\tilde{\mathcal{O}}$ onto $\tilde{\mathcal{V}}$, there exists $(\tau, g) \in \tilde{\mathcal{O}}$ such that 
\[\phi(\theta) + w = \psi^\theta(\tau) + g\]
i.e.
\begin{equation}\label{Nec cond-eq 1}
\psi^\theta(\tau) - \phi(\theta) = w - g .
\end{equation}

Let us denote $\tilde{\theta} = \rho_{\theta}(\tau)$ and let us show that Theorem \ref{local identifiability using Sigma1*}.ii) does not hold. By definition, $\phi(\tilde{\theta}) = \psi^\theta(\tau)$ and since $(\tau,g) \in \tilde{\mathcal{O}}$, $\| \phi(\theta) - \phi(\tilde{\theta}) \|_\infty < \epsilon$. Since $H \cap G = \{0\}$,  $w \in H$, $g \in G$ and $w \neq 0$, \eqref{Nec cond-eq 1} shows that
\[\phi(\tilde{\theta}) - \phi(\theta)  \neq 0.\]

Furthermore, since $w \in H \subset \Ker A(X, \theta)$ and $g \in G \subset \Ker A(X, \theta)$, \eqref{Nec cond-eq 1} shows that
\[\phi(\tilde{\theta}) - \phi(\theta) \in \Ker A(X, \theta),\]
so
\begin{equation*}
\begin{aligned}
\phi(\tilde{\theta}) & \in  N(X, \theta).
\end{aligned}
\end{equation*}

Summarizing, for any $\epsilon > 0$ there exists $\tilde{\theta} \in \ParS$ such that $\phi(\tilde{\theta})\in B_\infty(\phi(\theta), \epsilon) \cap  \Sigma_1^* \cap N(X, \theta) \backslash \{ \phi(\theta) \}$. The second item of Theorem \ref{local identifiability using Sigma1*} does not hold. Since it is equivalent, the first item of Theorem \ref{local identifiability using Sigma1*} does not hold either. In other words, the conclusion of Theorem \ref{Necessary condition} holds.
\end{proof}

\begin{proof}[Proof of Theorem \ref{Sufficient condition}]
Suppose that $C_S$ is satisfied. Using Lemma \ref{lemma dim(H)} and using $C_S$, we obtain
\begin{equation*}
\begin{aligned}
\dim(T \cap \Ker A(X, \theta)) = |F_\theta| + |B| - R_\Gamma = 0.
\end{aligned}
\end{equation*}
We thus have
\begin{equation}\label{T cap N = 0}
T \cap \Ker A(X,\theta) = \{0 \}.
\end{equation}

In order to apply Theorem \ref{local identifiability using Sigma1*}, let us show by contradiction that there exists $\epsilon > 0$ such that 
\begin{equation}\label{Sigma1* cap N = phi(theta)}
B_\infty(\phi(\theta), \epsilon) \cap \Sigma_1^* \cap N(X, \theta) = \{ \phi(\theta)\}.
\end{equation}
More precisely, we suppose that for all $n \in \NN^*$, there exists $\phi_n \in N(X, \theta) \cap \Sigma_1^*$ such that $\phi_n \neq \phi(\theta)$ and $\|\phi(\theta) - \phi_n\|_{\infty} < \frac{1}{n}$ and prove that it leads to $T \cap \Ker A(X, \theta) \neq \{0\}$, which contradicts \eqref{T cap N = 0}.

Using Proposition \ref{psi^t(U_t) is a neighborhood of t}, there exists $n_0 \in \NN^*$ such that for all $n \geq n_0$, there exists $\tau_n \in U_\theta$ such that $\phi_n = \psi^\theta(\tau_n)$. Since $\psi^\theta$ is a homeomorphism and $\psi^\theta(\thF)= \phi(\theta)$, 
\[\phi_n \rightarrow \phi(\theta)\]
implies that
\[ \tau_n \rightarrow \thF.\]
Moreover, for all $n \geq n_0$, $\tau_n \neq \thF$.

When $n$ tends to infinity, we can thus write
\[\phi_n - \phi(\theta)= \psi^\theta(\tau_n) - \psi^\theta(\thF) = D\psi^\theta(\thF) \cdot (\tau_n - \thF) + o(\tau_n - \thF).\]
Let us apply $A(X, \theta)$ and divide by $\| \tau_n - \thF\|$.
\begin{equation}\label{0 = AJ + o(1)}
\frac{1}{\| \tau_n - \thF\|}A(X, \theta) \cdot( \phi_n - \phi(\theta) )= A(X, \theta) \circ D\psi^\theta(\thF) \cdot \left(\frac{\tau_n - \thF}{\| \tau_n - \thF\|}\right) + \frac{1}{\| \tau_n - \thF\|}A(X, \theta)o\left(\tau_n - \thF\right). 
\end{equation}
Since $\phi_n \in N(X, \theta)$ for all $n \in \NN^*$,
\[\frac{1}{\| \tau_n - \thF\|}A(X, \theta) \cdot( \phi_n - \phi(\theta) )= 0.\]

Since $\frac{\tau_n - \thF}{\| \tau_n - \thF\|}$ belongs to the unit sphere, we can extract a subsequence that converges to a limit $h$ with norm equal to $1$.
Taking the limit in \eqref{0 = AJ + o(1)} according to this subsequence, we obtain
\[0 = A(X, \theta) \circ D\psi^\theta(\thF) \cdot h,\]
which shows that $D\psi^\theta (\thF) \cdot h  \in \Ker A(X, \theta)$. Since $h \neq 0$ and $D\psi^\theta (\thF)$ is injective, $D\psi^\theta (\thF) h \neq 0$ and
\[T \cap \Ker A(X, \theta) \neq \{ 0 \}.\]
This is in contradiction with \eqref{T cap N = 0}.

We have proven \eqref{Sigma1* cap N = phi(theta)}. We can now conclude thanks to Lemma \ref{local identifiability using Sigma1*}: there exists $\epsilon' > 0$ such that for any $\tilde{\theta} \in \Par$, if $\| \theta - \tilde{\theta}\| < \epsilon'$, then 
\[ f_\theta(X) = f_{\tilde{\theta}}(X) \quad \Longrightarrow \quad \theta \sim \tilde{\theta}.\]
\end{proof}

\section{Checking the conditions numerically}\label{Checking the conditions numerically-sec}

We restate and prove Proposition \ref{rank of A-main}.
\begin{prop}\label{rank of A}
Let $X \in \RR^{n \times V_0}$ and $\theta \in \Par$. We have
\[R_A = N_L \rk \left(\alpha(X,\theta)\right).\]
\end{prop}

\begin{proof}
Let $\eta \in \RR^{\mathcal{P} \times V_L}$. We have 
\[A(X, \theta) \cdot \eta = \alpha(X, \theta) \eta.\]
If we denote by $\eta^1, \dots, \eta^{N_L} \in \RR^{\mathcal{P}}$ the $N_L$ columns of $\eta$, the columns of $A(X, \theta) \cdot \eta$ are $\alpha(X, \theta) \eta^1, \dots, \alpha(X, \theta) \eta^{N_L}$. If we consider the matrix $\eta$ as a family of $N_L$ vectors of $\RR^{\mathcal{P}}$ and the matrix $A(X, \theta) \cdot \eta$ as a family of $N_L$ vectors of $\RR^{n}$, the operator $A(X, \theta)$ can then be equivalently described as
\[\fonction{A(X, \theta)}{(\RR^{\mathcal{P}})^{N_L}}{(\RR^{n})^{N_L}}{(\eta^1, \dots, \eta^{N_L})}{(\alpha(X, \theta) \eta^1, \dots, \alpha(X, \theta) \eta^{N_L}).} \]
The rank of such an operator is $N_L \rk(\alpha(X, \theta))$.
\end{proof}

We restate and prove Proposition \ref{Gamma is the differential-main}.
\begin{prop}\label{Gamma is the differential}
Let $X \in \RR^{n \times V_0}$ and $\theta \in \ParSX$. The function 
\[\varfonction{U_\theta}{\RR^{n \times V_L}}{\tau}{f_{\rho_\theta(\tau)}(X)}\]
is differentiable in a neighborhood of $\thF$, and we denote by $D_\tau f_{\rho_\theta(\tau_\theta)}(X)$ its differential at $\thF$. We have 
\begin{equation}\label{Gamma = D_tau f}
D_\tau f_{\rho_\theta(\tau_\theta)}(X) = \Gamma(X, \theta).
\end{equation}
\end{prop}

\begin{proof}
Using \eqref{fundamental prop of alpha-matrix version-main} at $\rho_\theta(\tau)$ and the definition of $\psi^\theta$ in \eqref{definition psi^theta-main}, we have
\[f_{\rho_\theta(\tau)}(X) = A(X, \theta) \cdot \psi^\theta(\tau).\]
Taking the differential of 
\[\varfonction{U_\theta}{\RR^{n \times V_L}}{\tau}{f_{\rho_\theta(\tau)}(X)}\]
at $\thF$, and using \eqref{definition of Gamma}, we obtain
\[
D_\tau f_{\rho_\theta(\tau_\theta)}(X)
=
A(X, \theta) \circ D \psi^\theta(\thF) = \Gamma(X, \theta).
\]
 
\end{proof}

To finish with, the following proposition gives explicit expressions of the coefficients of $\Gamma(X, \theta)$. These expressions are given for the sake of theoretical completeness. Note that when it comes to computing $\Gamma(X, \theta)$ in practice (in order to compute $R_\Gamma$), the correct approach is using backpropagation as described in Section \ref{Computations-sec-main} rather than evaluating the expressions in Proposition \ref{explicit expression of Gamma} which involve sums with very large numbers of summands.
\begin{prop}\label{explicit expression of Gamma}
If we decompose it in the canonical bases of $\ParF$ and $\RR^{\llbracket 1, n \rrbracket \times V_L}$, $\Gamma(X, \theta)$ is a $(n N_L) \times (|F_\theta| + |B|) $ matrix. For lighter notations, let us drop the dependency in $(X, \theta)$ and denote by $\gamma^{i, v_L}$ the lines of $\Gamma(X, \theta)$, for $i \in \lb 1, n \rb$ and $v_L \in V_L$, which satisfy $(\gamma^{i, v_L})^T \in \ParF$. For any $(i, v_L) \in \lb 1, n \rb \times V_L $, let us express the coefficients of $\gamma^{i, v_L}$, i.e. express $\gamma^{i,v_L}_{v_l \rightarrow v_{l+1}}$ for any $v_l \rightarrow v_{l+1} \in F_\theta$ and express $\gamma^{i, v_L}_{v_l}$ for any $v_l \in B$.
\begin{itemize}
    \item For any $l \in \lb 0,L-1 \rb$ and any $(v_l, v_{l+1}) \in V_l \times V_{l+1}$ such that $v_l \rightarrow v_{l+1} \in F_\theta $, 
\begin{multline}\label{expression of gamma_(v->v')}
\gamma^{i,v_L}_{v_l \rightarrow v_{l+1}} = \sum_{\substack{v_{0} \in V_{0} \\ \vdots \\ v_{l-1} \in V_{l-1}\\ v_{l+2} \in V_{l+2}\\  \vdots \\ v_{L-1} \in V_{L-1}}} x^i_{v_0}  \overline{w}_{v_0 \rightarrow v_1} \overline{a}_{v_l}(x^i,\theta) \prod_{\substack{1 \leq k \leq L-1 \\ k \neq l}} a_{v_k}(x^i, \theta) w_{v_k \rightarrow v_{k+1}}  \\ +  \sum_{l'=1}^L \sum_{\substack{v_{l'} \in V_{l'} \\ \vdots \\ v_{l-1} \in V_{l-1}\\ v_{l+2} \in V_{l+2}\\  \vdots \\ v_{L-1} \in V_{L-1}}} b_{v_{l'}} \overline{a}_{v_l}(x^i,\theta) \prod_{\substack{l' \leq k \leq L-1 \\ k \neq l}} a_{v_k}(x^i, \theta) w_{v_k \rightarrow v_{k+1}} ,\end{multline}
where $\overline{w}_{v_0 \rightarrow v_1} = w_{v_0 \rightarrow v_1}$ and $\overline{a}_{v_l}(x^i,\theta) = {a}_{v_l}(x^i,\theta)$ except when $l=0$ in which case $\overline{w}_{v_0 \rightarrow v_1} = 1$ and $\overline{a}_{v_l}(x^i,\theta) =1$.
\item For any $l \in \lb 1,L \rb$ and any $v_l \in V_l$,
\begin{equation}\label{expression of gamma_v}
\gamma^{i, v_L}_{v_l} = \sum_{\substack{v_{l+1} \in V_{l+1} \\ \vdots \\ v_{L-1} \in V_{L-1}}}  \prod_{l \leq k \leq L-1} a_{v_k}(x^i, \theta) w_{v_k \rightarrow v_{k+1}}.
\end{equation}
\end{itemize}

\end{prop}

\begin{proof}

Let $(i, v_L) \in \lb 1, n \rb \times V_L $. 

Let us compute $\gamma^{i, v_L}_{v_l \rightarrow v_{l+1}}$, for $l \in \lb 0,L-1 \rb$ and $(v_l, v_{l+1}) \in V_l \times V_{l+1}$ such that $v_l \rightarrow v_{l+1} \in F_\theta $. $\gamma^{i,v_L}_{v_l \rightarrow v_{l+1}}$ is the coefficient corresponding to the line $(i,v_L)$ and the column $(v_l \rightarrow v_{l+1})$ of $\Gamma(X, \theta)$. Let us denote by $h^{v_l \rightarrow v_{l+1}} \in \ParF$ the vector whose component indexed by $v_l \rightarrow v_{l+1}$ is equal to $1$ and whose other components are zero. Let us denote by $d^{i,v_L} \in \RR^{n \times V_L}$ the element whose component indexed by $(i,v_L)$ is equal to $1$ and whose other components are zero. Let us denote by $\langle \cdot , \cdot \rangle_{\RR^{n \times V_L}}$ the scalar product of the euclidean space $\RR^{n \times V_L}$. We have
\begin{align*}
\gamma^{i,v_L}_{v_l \rightarrow v_{l+1}} & = \left\langle d^{i,v_L}  \ , \ \Gamma(X, \theta) \cdot h^{v_l \rightarrow v_{l+1}} \right\rangle_{\RR^{n \times V_L}}\\
& = \left\langle d^{i,v_L} \ , \ A(X, \theta) \circ D\psi^\theta(\thF) \cdot h^{v_l \rightarrow v_{l+1}} \right\rangle_{\RR^{n \times V_L}} \\
& = \left\langle d^{i,v_L} \ , \  A(X, \theta) \cdot \frac{\partial \psi^\theta}{\partial \tau_{v_l \rightarrow v_{l+1}}}(\thF)  \right\rangle_{\RR^{n \times V_L}} \\
& = \left\langle d^{i,v_L} \ , \  \alpha(X, \theta)  \frac{\partial \psi^\theta}{\partial \tau_{v_l \rightarrow v_{l+1}}}(\thF)  \right\rangle_{\RR^{n \times V_L}} \\
& =  \left[\alpha(X,\theta)  \frac{\partial \psi^\theta}{\partial \tau_{v_l \rightarrow v_{l+1}}}(\thF) \right]_{i,v_L},
\end{align*}
where $ \left[\alpha(X,\theta)  \frac{\partial \psi^\theta}{\partial \tau_{v_l \rightarrow v_{l+1}}}(\thF) \right]_{i,v_L}$ denotes the coefficient $(i,v_L)$ of the product $\alpha(X,\theta)  \frac{\partial \psi^\theta}{\partial \tau_{v_l \rightarrow v_{l+1}}}(\thF)$.
Let us remind the dimensions in this product. For the left factor, recalling the definition given in the beginning of Section \ref{Conditions of local identifiability-sec}, we have $\alpha(X, \theta) \in \RR^{n \times \mathcal{P}}$. Concerning the right factor, since for any $\tau \in U_\theta$, we have $\psi^\theta(\tau) \in \RR^{\mathcal{P} \times V_L}$, the partial derivative satisfies $\frac{\partial \psi^\theta}{\partial \tau_{v_l \rightarrow v_{l+1}}}(\thF) \in \RR^{\mathcal{P} \times V_L}$. Hence, the dimension of the product is
\[\alpha(X,\theta)  \frac{\partial \psi^\theta}{\partial \tau_{v_l \rightarrow v_{l+1}}}(\thF) \in \RR^{n \times V_L}.\]
To obtain the coefficient $(i,v_L)$ of this product, we keep the $i^{th}$ line of the left factor, which is equal to $\alpha(x^i, \theta)$, and the column $v_L$ of the right factor, which is equal to $ \frac{\partial \psi^\theta_{v_L}}{\partial \tau_{v_l \rightarrow v_{l+1}}}(\thF)$. We thus have
\[\left[\alpha(X,\theta)  \frac{\partial \psi^\theta}{\partial \tau_{v_l \rightarrow v_{l+1}}}(\thF) \right]_{i,v_L} = \alpha(x^i, \theta) \frac{\partial \psi^\theta_{v_L}}{\partial \tau_{v_l \rightarrow v_{l+1}}}(\thF) \ = \ \sum_{p \in \mathcal{P}} \alpha_p(x^i, \theta) \frac{\partial \psi^\theta_{p,v_L}}{\partial \tau_{v_l \rightarrow v_{l+1}}}(\thF). \]

Let $p \in \mathcal{P}$. If $p = (v_0, \dots, v_L) \in \mathcal{P}_0$, looking at the case 1 in the proof of Proposition \ref{the differential of psi^theta is injective}, we have
\begin{align*}
    \frac{\partial \psi^\theta_{p,v_L}}{\partial \tau_{v_l \rightarrow v_{l+1}}}(\thF) = \One_{\{v_l \rightarrow v_{l+1} \in p\}} \ \prod_{\substack{k \in \lb 0 , L-1 \rb \\ k \neq l }} w_{v_k \rightarrow v_{k+1}}.
\end{align*}

Recalling the definition of $\alpha_p(x^i, \theta)$ in the case $p \in \mathcal{P}_0$, given in \eqref{definition of alpha}, we also have
\[\alpha_p(x^i,\theta) = x^i_{v_0} \prod_{k=1}^{L-1}a_{v_k}(x^i, \theta) ,\]
and thus
\begin{align}\label{proof expression of Gamma case 1}
   \alpha_p(x^i, \theta) \frac{\partial \psi^\theta_{p,v_L}}{\partial \tau_{v_l \rightarrow v_{l+1}}}(\thF) = \One_{\{v_l \rightarrow v_{l+1} \in p\}} \  x^i_{v_0} \prod_{k=1}^{L-1}a_{v_k}(x^i, \theta)  \prod_{\substack{k \in \lb 0 , L-1 \rb \\ k \neq l }} w_{v_k \rightarrow v_{k+1}}.
\end{align}
Now if $p = (v_{l'}, \dots, v_L) \in \mathcal{P}_{l'}$, for $l' \in \lb 1, \dots, L-1 \}$, looking at the case 2 in the proof of Proposition \ref{the differential of psi^theta is injective}, we have
\begin{align*}
    \frac{\partial \psi^\theta_{p,v_L}}{\partial \tau_{v_l \rightarrow v_{l+1}}}(\thF) = \One_{\{v_l \rightarrow v_{l+1} \in p\}} \ b_{v_{l'}} \prod_{\substack{k \in \lb l' , L-1 \rb \\ k \neq l }} w_{v_k \rightarrow v_{k+1}}.
\end{align*}

Recalling the definition of $\alpha_p(x^i, \theta)$ in the case $p \in \mathcal{P}_{l'}$, given in \eqref{definition of alpha}, we also have
\[\alpha_p(x^i,\theta) = \prod_{k=l'}^{L-1}a_{v_k}(x^i, \theta) ,\]
and thus
\begin{align}\label{proof expression of Gamma case 2}
   \alpha_p(x^i, \theta) \frac{\partial \psi^\theta_{p,v_L}}{\partial \tau_{v_l \rightarrow v_{l+1}}}(\thF) = \One_{\{v_l \rightarrow v_{l+1} \in p\}} \ b_{v_{l'}} \prod_{k=l'}^{L-1}a_{v_k}(x^i, \theta) \prod_{\substack{k \in \lb l' , L-1 \rb \\ k \neq l }} w_{v_k \rightarrow v_{k+1}}.
\end{align}

Finally, if $p=\beta$, looking at the case 3 in the proof of Proposition \ref{the differential of psi^theta is injective}, we have
\begin{align*}
    \frac{\partial \psi^\theta_{p,v_L}}{\partial \tau_{v_l \rightarrow v_{l+1}}}(\thF) = 0,
\end{align*}
and thus 
\begin{align}\label{proof expression of Gamma case 3}
   \alpha_p(x^i, \theta) \frac{\partial \psi^\theta_{p,v_L}}{\partial \tau_{v_l \rightarrow v_{l+1}}}(\thF) = 0.
\end{align}

Assembling \eqref{proof expression of Gamma case 1}, \eqref{proof expression of Gamma case 2} and \eqref{proof expression of Gamma case 3}, we can sum over all $p \in \mathcal{P}$, and obtain
\begin{multline*}
    \gamma^{i, v_L}_{v_{l+1} \rightarrow v_{l}} = \sum_{\substack{p \in \mathcal{P}_0 \\ p = (v_0, \dots, v_{L-1})}}  \One_{\{v_{l} \rightarrow v_{l+1} \in p\}} \ x^i_{v_0} \prod_{k=1}^{L-1}a_{v_k}(x^i, \theta)  \prod_{\substack{k \in \lb 0 , L-1 \rb \\ k \neq l }} w_{v_k \rightarrow v_{k+1}} \\ +\sum_{l'=1}^L \sum_{\substack{p \in \mathcal{P}_{l'} \\ p = (v_{l'}, \dots, v_{L-1})}} \One_{\{v_{l} \rightarrow v_{l+1} \in p\}} \ b_{v_{l'}} \prod_{k=l'}^{L-1}a_{v_k}(x^i, \theta) \prod_{\substack{k \in \lb l' , L-1 \rb \\ k \neq l }} w_{v_k \rightarrow v_{k+1}}
\end{multline*}
which can be reformulated, getting rid of the zero sums when $v_{l} \rightarrow v_{l+1} \not\in p$, as
\begin{multline*}
    \gamma^{i, v_L}_{v_{l+1} \rightarrow v_{l}} = \sum_{\substack{v_{0} \in V_{0} \\ \vdots \\ v_{l-1} \in V_{l-1}\\ v_{l+2} \in V_{l+2}\\  \vdots \\ v_{L-1} \in V_{L-1}}} x^i_{v_0} \overline{w}_{v_0 \rightarrow v_1} a_{v_{l}}(x^i,\theta) \prod_{\substack{k \in \lb 1 , L-1 \rb \\ k \neq l}} a_{v_k}(x^i,\theta) w_{v_k \rightarrow v_{k+1}}  \\ 
    + \sum_{l'=1}^L \sum_{\substack{v_{l'} \in V_{l'} \\ \vdots \\ v_{l-1} \in V_{l-1}\\ v_{l+2} \in V_{l+2}\\  \vdots \\ v_{L-1} \in V_{L-1}}} a_{v_{l}}(x^i,\theta)  b_{v_l'} \prod_{\substack{k \in \lb l', L-1 \rb \\ k \neq l}} a_{v_k}(x^i, \theta) w_{v_k \rightarrow v_{k+1}},
\end{multline*}
which shows \eqref{expression of gamma_(v->v')}.

The proof of \eqref{expression of gamma_v} is similar to the one of \eqref{expression of gamma_(v->v')}.
\end{proof}

\end{document}